\documentclass[a4paper,twosided, 11pt]{article}
\usepackage[utf8]{inputenc} 
\usepackage[T1]{fontenc} 
\usepackage[a4paper]{geometry} 
\usepackage{amsmath, amssymb, mathtools, amsthm}
\usepackage{mathtools} 
\usepackage{mathrsfs}  
\usepackage{graphicx} 
\usepackage{xcolor}   
\usepackage[english]{babel} 
\usepackage{bbm}
\usepackage[title]{appendix}
\usepackage{caption}
\usepackage{subcaption}
\usepackage{comment}
\usepackage{geometry}
\usepackage{cite}
\usepackage{float}
\usepackage{booktabs}
\usepackage{varwidth}
\usepackage[export]{adjustbox}
\usepackage{enumitem} 
\usepackage{hyperref} 

\newcommand \commentout[1] {}

\newcommand{\R}{\mathbb{R}}
\newcommand{\Td}{\mathbb{T}^d}
\newcommand{\N}{\mathbb{N}}

\newcommand {\Chi} {{\bf \raise 2pt \hbox{$\chi$}} }

\newcommand {\Div}  { {\rm div} }

\newcommand {\f}   {\frac}
\newcommand {\p}   {\partial}

\newcommand*{\dd}{\mathop{\kern0pt\mathrm{d}}\!{}}

\newcommand*{\DD}{\mathop{\kern0pt\mathrm{D}}\!{}}

\def\epsl{\varepsilon_l}

\DeclarePairedDelimiter{\abs}{\lvert}{\rvert}
\def\norm#1{\left\Vert #1\right\Vert}

\DeclareMathOperator*{\argmin}{argmin}

\DeclareMathOperator*{\esssup}{esssup}

\def\eps{\varepsilon}
\theoremstyle{plain}
\newtheorem*{thm*}{Theorem}
\newtheorem{thm}{\bf Theorem}[section]

\newtheorem{lemma}[thm]{Lemma}
\newtheorem{proposition}[thm]{Proposition}

\theoremstyle{remark}
\newtheorem{remark}[thm]{\bf Remark}
\newtheorem{definition}[thm]{\bf Definition}

\newcommand{\beq}{\begin{equation}}
\newcommand{\eeq}{\end{equation}}
\newcommand{\bea} {\begin{array}{rl}}
\newcommand{\eea} {\end{array}}
\newcommand{\bepa}{\left\{ \begin{array}{l}}
\newcommand{\eepa} {\end{array}\right.}
\newcommand{\diff}{\mathop{}\!\mathrm{d}}
\def\norm#1{\left\Vert#1\right\Vert}
\def\om{\Omega}

\numberwithin{equation}{section}

\title{Weak solutions and sharp interface limit of the anisotropic Cahn-Hilliard equation with disparate mobility and inhomogeneous potential}

\author{Charles Elbar\thanks{Université Claude Bernard Lyon 1, ICJ UMR5208, CNRS, Ecole Centrale de Lyon, INSA Lyon, Université Jean Monnet, 69622
Villeurbanne, France. Email: elbar@math.univ-lyon1.fr} 
\and Andrea Poiatti\thanks{University of Vienna, 1090 Vienna, Austria. Email: andrea.poiatti@univie.ac.at }
}

\date{}

\begin{document}

\maketitle

\begin{abstract}
We study the existence of weak solutions and the corresponding sharp interface limit of an anisotropic Cahn-Hilliard equation with disparate mobility, i.e., the mobility  is degenerate in one of the two pure phases, making the diffusion
in that phase vanish.  The double-well potential is polynomial and is weighted by a spatially inhomogeneous coefficient. In the limit when the parameter of the interface width tends to zero, and under an energy convergence assumption, we prove that the weak solutions converge to $BV$ solutions of a weighted anisotropic Hele-Shaw flow. 
\end{abstract}
\vskip .7cm

\noindent{\makebox[1in]\hrulefill}\newline

\tableofcontents

\section{Introduction}

The Cahn--Hilliard equation, introduced by J.\,W.~Cahn and J.\,E.~Hilliard in 1958~\cite{cahn_free_1958}, models the phase separation process in binary alloys, known as spinodal decomposition. Unlike models with sharp boundaries, the Cahn--Hilliard equation describes a diffuse interface between phases, whose width is controlled by a parameter \(\eps > 0\). As \(\eps\) tends to zero, interfaces sharpen, eventually forming evolving hypersurfaces. This makes the Cahn--Hilliard model a link between phase-field continuum models and Hele-Shaw (or Mullins-Sekerka) models, with applications in materials science, biology, and other fields.
Initially, the Cahn--Hilliard equation was assumed isotropic, i.e., enjoying uniform physical properties in all directions. However, many real-world systems are anisotropic, with properties like surface tension varying with interface orientation. This is evident in crystalline materials (see, e.g., \cite{CahnTaylor,KOBAYASHI1993410}), where atomic arrangements lead to direction-dependent interface evolution, forming facets or corners. The effect of anisotropies can also be visualized in the numerical simulations of the following Appendix \ref{numerical}. Another application of the anisotropic Cahn–Hilliard equation is to model the growth of thin solid
films, which has a role in the self-organization of nanostructures (cf. \cite{V1,V2}).
Anisotropy is also prevalent in biological contexts, such as tissue mechanics, where properties like stiffness depend on the alignment of structures like collagen fibers. In tumor growth, for instance, the interaction between cancer cells and the extracellular matrix can result in anisotropic tumor shapes~\cite{cristini2009nonlinear}. 

Incorporating anisotropy into the Cahn--Hilliard equation leads to the \textit{anisotropic Cahn--Hilliard equation}. When the mobility is degenerate in one of the pure phases, taking the sharp-interface limit (\(\eps \to 0\)) in this setting results in an anisotropic Hele-Shaw flow, accounting for direction-dependent effects. The presence of a space dependent potential then leads to a \textit{weighted} anisotropic Hele-Shaw flow.

In this work, we investigate the anisotropic Cahn--Hilliard equation with disparate mobility and space dependent potential on a $d$-dimensional torus \(\Omega=\mathbb T^d\), $d=2,3$. The model reads:

\begin{equation}\label{eq:CH}
\partial_t u = \Div(u\nabla\mu), \quad \mu = -\eps\Div(A_{p}(x,\nabla u)) + \frac{1}{\eps} F_{u}(x,u) \quad \text{in } (0,T)\times\Omega,
\end{equation}
or alternatively, introducing the flux $j$, 
\begin{equation}\label{eq:CHj}
\partial_t u = \Div\ j, \quad j=-u\nabla\mu = u\nabla\left(\eps\Div(A_{p}(x,\nabla u)) - \frac{1}{\eps} F_{u}(x,u)\right) \quad \text{in } (0,T)\times\Omega.
\end{equation}

Here, \(T > 0\) is the final time. The parameter \(\eps > 0\) relates to the interface thickness. The function \(A = A(x,p)\) is \(C^1\) in both variables, two-homogeneous in \(p\), and satisfies strong convexity and growth conditions (see Assumption~\eqref{ass:A}). The potential \(F(x,u)\) is a polynomial double-well potential depending on \((x,u)\) and satisfying certain assumptions (see Assumption~\eqref{ass:double_well}), namely it can be factorized as $K(x)W(u)$, where $W(u)$ is the standard double well quartic potential (which is zero at the pure phases $u=0,1$) and $K$ is a $C^1$ strictly positive weight, representing some heterogeneity in the material. The notation \(A_p\) and \(A_x\) denotes the gradients of \(A(x,p)\) with respect to \(p\) and \(x\), respectively, with similar definitions for \(F_u\) and \(F_x\). Observe that the mobility $u$ in front of $\nabla \mu$ degenerates only in one phase, that is $u=0$. This is different from the standard degenerate mobility case, in which the mobility degenerates in both the phases $u=0$ \textit{and} $u=1$.

Our main goal is to prove the existence of global weak solutions for \eqref{eq:CH} under the stated assumptions and to analyze the sharp interface limit as \(\eps \to 0\). In particular, we show that, assuming an appropriate energy convergence assumption, a weak solution of the anisotropic Cahn--Hilliard equation converges to a solution of the anisotropic weighted Hele--Shaw flow. We also add some numerical experiments to study the effects of anisotropy in the Cahn-Hilliard equation.

System \eqref{eq:CH} comes with two functionals, useful to obtain suitable estimates. First of all, the system can be viewed as the Wasserstein gradient flow of the energy functional
\begin{equation}
\label{eq:energy}
E[u] :=\int_{\Omega}\eps \,A (x,\nabla u(x) ) + \f{1}{\eps}F (x,u(x) )\diff x.
\end{equation}
In particular the solutions of~\eqref{eq:CH} decrease along the energy. 
A key auxiliary functional we also use is the so-called \emph{entropy}:
\begin{equation}
\label{eq:entropy}
\Phi(u):= \int_{\Omega} u(x)\log u(x)\diff x.
\end{equation}

Energy and entropy formally satisfy the relations

\begin{align}
&\label{energyest}\f{dE}{dt} + \int_{\Omega}u\,|\nabla \mu|^2\diff x=0,\\
&\f{d\Phi}{dt}+ \int_{\Omega}(\varepsilon\Div(A_{p}(x,\nabla u))- \frac1\eps F_{u}(x,u))\,\Delta u\diff x =0. 
\end{align}

The previous identities are crucial to find \textit{a priori} estimates and prove existence of weak solutions and the sharp interface limit, as these two results are based on compactness arguments.

Concerning the sharp interface limit as $\eps\to0$, given a solution $(u_\eps,j_\eps)$ satisfying for any $   \eps>0$ equations \eqref{eq:CH} in suitable weak form, the main ingredient is the energy estimate \eqref{energyest}, which allows to control the main quantity through the well-known Modica–Mortola trick \cite{Modica}, namely, defining  
\begin{equation}
\label{eq:definition_psi1}
    \psi(s) = 2 \int_{0}^{s} \sqrt{W(r)} \,\diff r
\end{equation}
we can control in $L^1$ the gradient of $\psi\circ u_\eps$ by means of the energy inequality \eqref{energyest}. This allows to obtain the desired compactness for passing to the limit.

Another essential ingredient is the observation that from the second identity in \eqref{eq:CHj}, after some simple computations, we have 
\begin{align}
 \nonumber j_\eps&= -\text{div}\ \mathbf T_\eps+\frac1\eps F_x(x,u_\eps)+\eps A_x(x,\nabla u_\eps)\\&\quad +\nabla\left(\eps\nabla u_\eps \cdot A_{p}(x,\nabla u_\eps)+\f{1}{\eps}u_\eps F_{u}(x,u_\eps)-\text{div}(\eps u_\eps A_p(x,\nabla u_\eps))\right),\label{flux1}
\end{align}
where $\mathbf T_\eps$ denotes the anisotropic energy stress tensor
\begin{align}
    \mathbf T_\eps:= \left(\eps A(x,\nabla u_\eps)+\frac 1 {\eps}F(x,u_\eps)\right)Id-\eps\left(\nabla u_\eps \otimes  A_p(x,\nabla u_\eps)\right).\label{stresstensor}
\end{align}
Observe that the right-hand side of \eqref{flux1}, excluding the terms with $F_x$ and $A_x$ (which indeed are less standard and need to be treated more carefully), is  in divergence form. This allows to test by smooth test functions and integrate by parts. The resulting weak
formulation has then good compactness properties, so that given a sequence of weak solutions, one only has to guarantee energy convergence to prove that the limit is a weak solution. Indeed, one only needs to pass to the limit in first-order terms.

In conclusion, to pass to the limit as $\eps\to0$ in the weak formulation of \eqref{flux1} we will exploit two different methods, i.e., by means of a suitable version of Reshetnyak continuity theorem (similar to \cite{CicaleseNagasePisante}) and by means of an anisotropic tilt excess approach (in the spirit of \cite{laux_kroemer,LauxUllrich,laux2018convergence}).

We now review in the section below the existing literature on the Cahn--Hilliard equation and the study of its sharp interface limit.

\subsection{Literature review}

To get a better understanding of the results presented here, we begin by revisiting the well-known isotropic Cahn-Hilliard equation, which can be seen as a particular case of our study when 
\[
A(x,\nabla u) = \frac{1}{2} \lvert \nabla u \rvert^2 
\quad \text{and} \quad 
F(x,u) = u^2(1 - u)^2.
\]
This isotropic model has been extensively studied, providing a solid basis of techniques that we adapt to the anisotropic setting. We first summarize results on the well-posedness of the isotropic Cahn-Hilliard equation and then revisit the literature on its sharp interface limit. Although the anisotropic version of the Cahn-Hilliard equation is relatively new, it has been rigorously investigated in some specific cases. In particular, the framework on the anisotropy that we use here is based on previous research. Finally, we mention some results on the anisotropic Allen-Cahn equation, which shares the same energy functional but is the gradient flow with respect to a different metric (\(L^2\) distance instead of the Wasserstein distance adopted in our setting).

\subsubsection{Isotropic Cahn-Hilliard equation: well posedness}

The isotropic Cahn--Hilliard equation is written in its simplest form as
$$
\partial_t u - \Div (m(u)\,\nabla\mu ) = 0, 
\quad 
\mu = -\eps\,\Delta u + \frac1\eps F'(u).
$$

This equation has now been extensively studied and a series of works has proven well-posedness results under various assumptions on \(m(u)\) and \(F(u)\).

When the mobility is constant, i.e., \(m(u) \equiv 1\), the equation simplifies considerably. In that case, the classical proofs often rely on a Galerkin scheme for proving existence. Typically, one constructs approximate solutions by projecting the PDE onto finite-dimensional subspaces, and then use compactness arguments to pass to the limit. Under these non-degenerate conditions, uniqueness is usually proven via a Gronwall-type argument.

A more delicate situation is when the mobility function can vanish for some values of \(u\). The equation is referred to as the \emph{Cahn--Hilliard equation with degenerate mobility} when the mobility degenerates in both the pure phases $u=0,1$, and \emph{Cahn--Hilliard equation with disparate mobility} when the mobility degenerates only in one pure phase, say $u=0$. The pioneering work of Elliott and Garcke~\cite{elliott_cahn-hilliard_1996} proved the existence of weak solutions in this degenerate setting, though the uniqueness problem remains largely open.

Regarding the choice of potential \(F\), which usually does not depend on space, two general categories are widely studied: smooth double-well potentials where \(F\) is a polynomial-function with two minima~\cite{mr3448925}, singular double-well potentials which might involve logarithmic terms, or other singularities to ensure strict phase separation~\cite{elliott_cahn-hilliard_1996}. We refer to \cite{MZ,GPoiatti} for some updated results about the instantaneous strict separation from pure phases in the case of constant mobility and singular potential. In conclusion, note that, the two-phase (isotropic or anisotropic) Cahn-Hilliard equation has been also extended to the case of multi-component mixtures (see, e.g., \cite{EL,GGPS, G1,G2}).

\vspace{6pt}





\vspace{6pt}

\noindent
\textbf{Gradient Flow Approaches and JKO Scheme.}
Another perspective to prove well-posedness of the Cahn-Hilliard equation is to view it as a gradient flow of an energy functional. In the case of the disparate mobility $m(u)=u$, this is a gradient flow in the Wasserstein metric $W_2$. In particular, one uses the well-known Jordan--Kinderlehrer--Otto (JKO) scheme to construct a time-discrete approximation and then pass to the limit where the discretization parameter is sent to 0. For general concave mobilities $m(u)$, adapting the JKO scheme requires considering modified Wasserstein distances introduced in~\cite{mr2448650}. The JKO approach has been followed in many works including~\cite{lisini-ch-gradient-flow,doi:10.1142/S021919972450041X,antonio2024competing,laux_kroemer}.

\subsubsection{Isotropic Cahn-Hilliard equation: sharp interface limit}

Consider the minimization problem 
$$
\min\left\{E_{\eps}[u]: u\ge 0, \, \int_{\Omega} u=m\right\},
$$
where $E_{\eps}[u]$ is the energy defined in~\eqref{eq:energy}. In the isotropic case, if we let $u_{\eps}$ be a minimizer of the previous problem, and $\lambda_{\eps}$ be the Lagrange multiplier associated to the mass constraint we obtain:
$$
\lambda_{\eps} = -\eps\Delta u_{\eps} +\f{1}{\eps} F'(u_{\eps}).
$$

$\lambda_{\eps}$ is often referred as the chemical potential. In this setting, it was conjectured by Gurtin in~\cite{gurtin1987some} and then proven by Luckhaus and Modica in~\cite{Luckhaus1989TheGR} that $\lambda_{\eps}\to \lambda$, where $\lambda$ is related to the constant sum of principal curvatures of the interface. This result corresponds to the Gibbs-Thomson relation for surface tension. Incorporating anisotropy and spatial heterogeneity in the gradient term, Cicalese et. al.~\cite{CicaleseNagasePisante} and Garcke, Kraus~\cite{garcke2009anisotropic} showed independently that a similar result holds, yielding an anisotropic Gibbs-Thomson relation. We also refer to~\cite{garcke2011existence} for the Stefan problem.

These results are at the basis of the sharp interface limit we aim to study. Indeed, in the Cahn-Hilliard equation the chemical potential is defined as 
$$
\mu_{\eps} = - \eps \Div(A_{p}(x,\nabla u_{\eps})) + \f{1}{\eps}F_{u}(x,u_{\eps})= -\eps \Delta u_{\eps} + \f{1}{\eps}F'(u_{\eps}) \text{ (in the isotropic case)}.
$$

In this case $\mu_{\eps}$ is not a Lagrange multiplier anymore and therefore it is not a constant, but rather a function. It is the first variation of the the energy functional $E_{\eps}[u_{\eps}]$. However many ideas already applied in~\cite{Luckhaus1989TheGR} and~\cite{CicaleseNagasePisante} still apply here, as shown in \cite{laux2018convergence}.

Concerning the sharp interface limit of the Cahn-Hilliard equation, a classical starting point is the work of Alikakos~\cite{alikakos1994convergence}, where the author considered the Cahn--Hilliard equation with constant mobility and proved rigorously the convergence result $\eps\to 0$ toward classical solutions of the Mullins-Sekerka flow. Later, Chen~\cite{chen1996global} proved the result using a varifold framework.

\bigskip

The idea behind the varifold setting is quite natural once one notices that, in some cases, there may be hidden interfaces \emph{inside} a region where the limiting solution is constant. Specifically, one might see an interface with a positive contribution to the total energy, even though the phase-field variable remains at (say) $1$ on both sides. These hidden or \emph{phantom interfaces} are not captured by a strict characteristic-function, or $BV$, description. This suggests that $BV$ solutions may not be the best framework to understand all the scenarios taking place as $\eps \to 0$. Hence, one is naturally led to refined solution concepts, such as those using varifolds or suitably weaker formulations of the moving interface. 

\bigskip

On the other hand, if one is willing to impose additional conditions ruling out phantom interfaces, and in general any loss of interface measure, one may work under an energy convergence assumption (see, for example, the energy condition in~\eqref{energ} that we impose) to obtain stronger notions of weak solutions. This assumption is quite classical and has been used in a different series of works~\cite{luckhaus1995implicit,hensel2021bv,laux2018convergence,kim2024density, LauxUllrich,laux_kroemer}. In the work~\cite{laux_kroemer}, for instance, such an assumption allows to recover a stronger weak solution (BV solution) description for the (isotropic) Hele--Shaw flow in the sharp interface limit. We follow a similar line of reasoning here.

\bigskip

A natural question is if it is possible to drop the assumption on the energy convergence and still obtain a limit in the varifold sense. This is left as an  open problem: in the isotropic Cahn--Hilliard setting, the usual strategy is to show that the so-called \emph{discrepancy measure}
$$
\xi^{\eps}(u_{\eps}) = \frac{\eps}{2}\, \lvert \nabla u_{\eps} \rvert^{2} - \frac{1}{\eps}\,F (u_{\eps} ),
$$
is non-positive, as first introduced in \cite{chen1996global}. However, in the anisotropic framework, one might consider
$$
\xi^{\eps}(u_{\eps}) = \eps \,A (x,\nabla u_{\eps} ) - \frac{1}{\eps}\,F (x,u_{\eps} ),
$$
where $A$ is a (possibly non-constant) matrix describing the interfacial energy. Even in the simpler situation where $A$ is constant, it seems complicated to show that this discrepancy measure remains non-positive, preventing us from proving the convergence in a purely varifold setting free of extra hypotheses. We also refer to~\cite{lee2016sharp,antonopoulou2018sharp,lam2019sharp} for works on the study of the sharp interface limit of the isotropic Cahn-Hilliard equation. Finally, we mention that another link between Cahn-Hilliard and free boundary models can also be obtained via incompressible limits (see, e.g., \cite{EPPS,crmeca}).

\subsubsection{The anisotropic Cahn-Hilliard and Allen-Cahn equations}

The anisotropic Cahn-Hilliard equation has been used to study snow crystal growth, solidification of metals and self organization of nanostructures in ~\cite{barrett2013stable,barrett2014stable,dziwnik2016existence,dziwnik2017anisotropic,V1}.
The well-posedness has been tackled in different papers. We refer to~\cite{garcke2023anisotropic} for the well-posedness and further properties in the setting of constant mobility. In the degenerate mobility setting, and for a particular anisotropic energy, we refer to~\cite{dziwnik2016existence}. Finally, we mention ~\cite{Patrik}, concerning the existence of weak solutions with degenerate mobility in the setting of, and with a similar approach to, Elliott and Garcke~\cite{elliott_cahn-hilliard_1996}. 

We also mention the anisotropic Allen-Cahn equation, which is the $L^2$ gradient flow of the anistropic Ginzburg Landau energy~\eqref{eq:energy}. This equation has been studied in~\cite{LAUT:2020} in the anisotropic setting.

In conclusion, concerning the spatial dependence of the potential $F$, here we assume that $F$ factorizes in the product of two terms $K(x)W(u)$, the first factor depending on $x$ and the second one depending on the phase-field $u$.  This approach has been adopted in many works, especially concerning the sharp interface limit of some variants of Allen-Cahn equation (see, for instance, \cite{Matano, Qi,alfaromotion}). More general inhomogeneous potentials have been taken into consideration in \cite{Bouchitte_Gamma}. The technique introduced in the former has been extended to the sharp interface limit of an Allen-Cahn equation with spatial dependent double-well potential, under suitable assumptions, in~\cite{ganedi2024convergence}.  Here the authors prove that, under the usual energy convergence assumption, the model converges to a weighted mean curvature flow in the sharp interface limit. 
We observe that so far we are not aware of any result in the literature concerning the sharp interface limit of the Cahn-Hilliard equation with degenerate (or disparate) mobility and spatial dependent double well potential. Our result is thus also new in this direction.

\subsection*{Contents of the paper}
In the next section, we detail our notations and the functional settings. We then introduce the Finsler metric framework, which is commonly used to study anisotropic problems, and provide technical results related to \(BV\) functions. Following this, we present our main theorems on the existence of weak solutions and the convergence to a weighted anisotropic Hele-Shaw flow as \(\eps \to 0\). 
Section~\ref{sect:existence} is dedicated to the existence of weak solutions, that is the proof of Theorem~\ref{thm:weak_sol_degenerate}. The proof is achieved with the JKO scheme. In the first subsection we recall some preliminaries and useful tools in optimal transport and JKO scheme and the second subsection is dedicated to the proof. Section~\ref{sect:sharp_interface} is dedicated to the sharp interface limit under a standard energy convergence assumption, that is the proof of Theorem~\ref{thm:sharp_interface}. We prove this theorem using two different approaches, by means of a suitable version of Reshetnyak continuity theorem and by means of an anisotropic tilt excess approach. In conclusion, in Appendix \ref{Appendix} we prove a lemma related to the anisotropic version of Reshetnyak continuity theorem, whereas Appendix \ref{numerical} is dedicated to some numerical simulation.

\section{Preliminaries and main results}

\label{prelr}

\subsection{Notation and functional settings}
We assume, for the sake of simplicity, that $\Omega$ is the \(d\)-dimensional torus $\mathbb T^d$ (or, more generally, any domain where periodic boundary conditions are imposed). We denote by $\mathcal L^d$ the $d$-dimensional Lebesgue measure. On the other hand $\mathcal H^m$ is the $m$-dimensional Hausdorff measure. We then denote by $C^k(\om)$, $k\in\N\cup\{\infty\}$, the space of continuous functions in $\om$ which are $k$-times continuously differentiable, as well as by $C^k_c(\om)$ the space of continuous functions with compact support in $\om$ which are $k$-times continuously differentiable. The Sobolev spaces are denoted as usual by $W^{k,p}(\Omega )$%
	, where $k\in \mathbb{N}$ and $1\leq p\leq \infty $, with norm $\Vert \cdot
	\Vert _{W^{k,p}}$. The Hilbert space $W^{k,2}(\Omega )$ is indicated by $H^{k}(\Omega )$ with norm $\Vert \cdot \Vert _{H^{k}}$. Note that we additionally denote, for $k=0$, $\norm{\cdot}=\norm{\cdot}_{L^2}$. Also, for fractional Sobolev spaces we write $H^s(\om)$ for any $s\in(0,1)$, to indicate the fractional space $W^{s,2}(\Omega)$.
Let now $X$ be a Banach space. We denote by $L^q (a,b; X )$, $0\leq a\leq b$, $q\in[1,\infty]$, the Bochner space of $X$-valued $q$-integrable (or essentially bounded functions). Moreover, given a generic interval $J$, the function space 
$C_c^\infty (J ; X)$ denotes the vector space of all $C^\infty$-functions $f : J\to X$ with compact support in $J$.  
\subsection{Finsler metric}

\label{phi}
 We define a map
$$
    \phi  :  \Omega \times \mathbb{R}^d  \to  [0,+\infty),
$$
    which is strictly convex, in the sense that, for any $x \in \Omega$, the map $\phi^2(x,\cdot)$ is strictly
convex on $\R^d$. Also assume that $\phi$ is continuous and satisfies the following properties:
\begin{align}
    &\phi (x, t\,\xi ) 
      = \lvert t\rvert\,\phi (x, \xi )
      \quad 
      \text{for all }x \in \Omega,  \xi \in \mathbb{R}^d,  t \in \mathbb{R},
      \label{prop_Finsler1}
    \\
    &\lambda\,\lvert \xi\rvert 
      \le \phi (x,\xi ) 
      \le \Lambda\,\lvert \xi\rvert
      \quad
      \text{for some }0<\lambda\le\Lambda<\infty.
      \label{prop_Finsler2}
\end{align}
These two conditions express the positive $1$-homogeneity of $\phi$ in the second argument and give uniform bounds, so that $\phi(\cdot,\cdot)$ defines a Finsler norm (with $x$-dependence).

Let $B_{\phi(x)}$ be the (convex) unit ball associated to $\phi$, at a point $x\in\Omega$, that is
$$
B_{\phi(x)} = \{\xi\in \R^d, \, \phi(x,\xi)\le 1\}. 
$$

The dual Finsler function $\phi^{\circ}\colon \Omega\times \mathbb{R}^d\to [0,+\infty)$ is defined by
$$
    \phi^{\circ} (x,\xi^{\ast} )
     =
    \sup 
     \{\,\xi^{\ast}\cdot \xi 
    \,\colon\, \xi\in B_{\phi(x)} \},
$$
and we let $B_{\phi^\circ(x)}$ be the convex unit ball associated to $\phi^\circ$, at a point $x\in\Omega$, that is
$$
B_{\phi^\circ(x)} = \{\xi^*\in \R^d, \, \phi^\circ(x,\xi^*)\le 1\}. 
$$

One may verify that $\phi^{\circ}$ also satisfies properties \eqref{prop_Finsler1}--\eqref{prop_Finsler2}, and $\phi^{\circ\circ}$ in general equals the convex envelope of $\phi$. For a vector $\nu\in S^{d-1}$, the $\phi$-vector $\nu_{\phi}(x)$ and $n_{\phi}(x)$ are defined as 

\begin{align}
\nu_{\phi}(x) = \f{\nu}{\phi^{\circ}(x,\nu)}, \quad n_{\phi}(x) = \phi^{\circ}_{\xi}(x,\nu_{\phi}),  
\label{normals}
\end{align}
where $\phi^{\circ}_{\xi}$ denotes the partial derivative of $\phi^{\circ}$ with respect to the vector $\xi$. Formally, $\nu_{\phi}(x)$ rescales the Euclidean unit normal so that it becomes a Finsler unit normal, while $n_{\phi}(x)$ which is known as Cahn–Hoffmann vector, can be seen as the corresponding dual vector in the co-tangent space.
Observe that the following elementary properties can be proven (see, e.g., \cite[Section 2.1]{Finsler_geometry}):

\begin{lemma}[Elementary Properties of \(\phi\) and \(\phi^\circ\)]
\label{elementary}
For each $x\in \Omega$ and for all $\xi,\xi^*\in \mathbb{R}^d\setminus\{0\}$, it holds:
\begin{enumerate}
\item[\emph{(i)}] 
\(\displaystyle
    \phi^{\circ}_\xi (x,t\xi^* )
    = \frac{t}{\lvert t\rvert}\,\phi^{\circ}_\xi (x,\xi^* ),
    \quad
    \phi^\circ_{\xi\xi} (x,t\xi^* )
    = \frac{1}{\lvert t\rvert}\,\phi^\circ_{\xi\xi} (x,\xi^* ),
\)
for all $t\neq 0$.

\item[\emph{(ii)}] 
\(\displaystyle
    \phi (x,\xi )
     = 
    \phi_\xi (x,\xi )\,\cdot\,\xi,
    \quad
    \phi^\circ (x,\xi^* )
     = 
    \phi^\circ_\xi (x,\xi^* )\,\cdot\,\xi^*.
\)

\item[\emph{(iii)}] 
\(\displaystyle
    \phi (x,\phi^\circ_\xi(x,\xi^*) ) 
     = 
    \phi^\circ (x,\phi_\xi(x,\xi) )
     = 
    1.
\)
\end{enumerate}

\noindent
In particular,
$$
    \nu_\phi(x)\cdot n_\phi(x)
     = 
    1,
$$
which follows immediately upon noting that
\(\nu_\phi(x)\in \partial B_{\phi^\circ(x)}\) is precisely the dual direction to $n_\phi(x)\in \partial B_{\phi(x)} $.
\end{lemma}

Consider now $E\subset \R^d$ with $C^2$ boundary, and define, for any $x\in \partial E$, $\nu(x)$ as the unit inner normal to $\partial E$ at $x$. For a given $C^1$ vector field $X:\partial E\to \R^d$ we define the $\phi-tangential$ divergence of $X$ on $\partial E$ (see \cite{CicaleseNagasePisante}) as    \begin{align*}
    \text{div}_\phi X= \text{tr}\left[(Id-n_\phi\otimes \nu_\phi)\nabla \tilde{X}+\phi^\circ_x(x,\nu_\phi)\otimes \tilde{X}\right],
\end{align*} 
where $\tilde{X}$ is a smooth extension of X to a neighborhood of $\partial E$. 
Extending also to a neighborhood of $\partial E$ the vector fields $\nu_\phi$ and $n_\phi$ by regular fields
without relabeling, we can introduce the $\phi$-mean curvature
$H_\phi$ of $\partial E$ as
\begin{align*}
    H_\phi=-\text{div}_\phi n_\phi.
\end{align*}
In conclusion, as in \cite{CicaleseNagasePisante}, we report here a theorem which will be useful in showing that the weak formulation of the Hele-Shaw flow obtained in the sharp interface limit is coherent with the strong one, as long as we assume a smooth solution. For a proof we refer to  (2.3) and (3.2) in \cite{FragalaBellettini}.
\begin{thm}
\label{basic1}
   Let $E\subset \R^d$ be an open set with $C^2$ boundary. Let $U\subset 
 \R^d$ be a neighborhood of $\partial E$ and $g\in C_c^1(U;\R^d)$. Then
 \begin{align*}
     \int_{\partial E}H_\phi \nu_\phi \cdot g\, \phi^{\circ}(x,\nu)\diff\mathcal H^{d-1}=-\int_{\partial E}\Div_\phi g\,\phi^{\circ}(x,\nu)\diff\mathcal{H}^{d-1}.
 \end{align*}
\end{thm}
\subsection{BV functions and anisotropic perimeters}
We recall here the basic definitions of BV functions and the notion of anisotropic
perimeter. Given a vector-valued measure $\mu$ on $\Omega$, we denote by $\vert\mu\vert$ its total variation and
we adopt the notation $\mathcal M(\Omega)$ ($\mathcal M^+(\om)$, respectively) for the set of all signed (positive, respectively) Radon measures on $\Omega$ with bounded
total variation. The Lebesgue measure of a set $E$ is indicated by $\vert E\vert$. 
Recall that $u \in L^1(\Omega)$ belongs to the space $BV (\Omega)$ of functions of bounded
variation if its distributional derivative $Du$ belongs to $\mathcal M(\Omega)$, where we denote by $Du$ the
$\mathbb R
^d$-valued measure whose components are $D_1u,\ldots ,D_d u$. A set $E$ is of finite perimeter in $\Omega$ if its characteristic function $\chi_E \in
BV (\Omega)$ and we denote by $\mathcal{P}_\om (E) = \vert D \chi_E\vert(\Omega)$ the perimeter of E in $\Omega$. The family
of sets of finite perimeter can be identified with the functions $u \in BV (\Omega; \{0, 1\})$. It holds that
for $E = \{x\in \om : \, u(x) = 1\}$,
$$\mathcal P_\Omega(E) = \vert Du\vert (\Omega) = \mathcal H^{d-1}
(\partial^*E\cap  \Omega),$$
where $\partial^*E$ is the reduced boundary of $E$.
We now revise the definitions and some properties of the anisotropic total variation
for $BV$-functions and introduce the anisotropic perimeter. We start with the definition of the anisotropic $\phi$-total
variation (where $\phi$ is defined in Section \ref{phi}) of the $d$-dimensional measure $\mu\in\mathcal M^d(\Omega)$ as 
\begin{align*}
\abs{\mu}_{\phi}(\Omega):=\sup\left\{\int_\Omega 
 \sigma\cdot  \diff\mu ;\ \sigma\in C_c
(\Omega; \mathbb R^d ) , \, \sigma(x) \in B_{\phi(x)}\right\}.
\end{align*}
Let $u\in BV(\Omega)$. The anisotropic $\phi$-total variation of its (weak) gradient $Du$ is then
\begin{align}\label{variation}
\vert Du\vert_{\phi}(\Omega) = \sup\left\{\int_\Omega 
u\,  \Div \sigma\, \diff x ;\ \sigma\in C^
1_c
(\Omega; \mathbb R^d ) ,\,  \sigma(x) \in B_{\phi(x)}\right\}.
\end{align}
We point out here that it would be more intuitive to write $\vert \cdot\vert_{\phi^\circ}$, as the dual functions satisfy a constraint with respect to $\phi$. Nevertheless the notation adopted here is well established in the literature (see, e.g., \cite{CicaleseNagasePisante, 2CNP}).
Note that by the hypotheses on $\phi$, from Theorem 5.1 in \cite{2CNP} we have that
the $\phi$-total variation is $L^1
(\Omega)$-lower semicontinuous and admits the following integral
representation
$$\vert Du\vert_\phi(\Omega) = \int_\Omega \phi^{\circ}(x,\nu_u)d\vert Du\vert,\quad \forall u\in BV(\Omega),$$ where $\nu_{u}:=\frac{Du}{\vert Du\vert}$, is the Radon-Nikodym derivative of $Du$ with respect to its total variation.
Clearly, if $\phi^{\circ}(x,\xi) = |\xi|$ then the $\phi$-total variation $\vert Du\vert_\phi(\Omega)$ agrees with $\vert Du\vert(\Omega)$.
We now set the definition and some properties of anisotropic perimeters. Let $E \subset\mathbb R^d$ be a set of finite perimeter in $\Omega$. Given $K\in C(\overline{\Omega})$ such that $\inf_{x\in \Omega}K(x)\geq K_*>0$, we define the weighted $\phi$-anisotropic perimeter of $E$ in $\Omega$ as
\begin{align}\label{perimeter}
\mathcal{P}^K_\phi(E) = \int_{\partial^*E\cap \Omega}\sqrt{K(x)}\phi^{\circ}
(x,\nu(x)) d\mathcal H^{d-1},
\end{align}
where $\nu$ is the measure theoretic unit inner normal to $\partial^*E$.
Observe that it holds 
$\mathcal P^K_\phi(E) = \vert \sqrt{K(\cdot)}D\chi_E\vert_\phi(\Omega)=\sqrt{K(\cdot)} \vert D\chi_E\vert_\phi(\Omega)$.

To conclude this section, we present some useful lemmas concerning geometric measure theory. First, we observe that we can state the following result, which can be proven in the very same way as in \cite[Lemma 2.5]{LauxUllrich}, thanks to Lemma \ref{elementary}:
\begin{lemma}\label{sup}
Let  $G\in L^1
(\Omega)^d$, and let $\phi$ be as in Section \ref{phi}. Then
$$
\int_\Omega
\phi^{\circ}(x,G)\diff x = \sup_{\eta}\left\{
\int_\Omega G \cdot\eta\, \diff x\right\},$$
where the supremum is taken over all $\eta \in C
^1
(\Omega)^d$
such that $\eta(x)\in B_\phi(x)$ for any $x\in \Omega$. 
\end{lemma}
Another technical lemma concerns an analogue of the weak*  convergence for BV functions depending on time. Following the notation in \cite[Definition 2.27]{Ambrosiofuscopallara}, given a function $v\in L^1(0,T;\mathcal{M}(\Omega))$, we set $\mathcal L^1_{|(0,T)}\otimes v(t)$ the measure such that
\begin{align*}
    \mathcal L^1_{|(0,T)}\otimes v(t)(B):=\int_0^T\left(\int_\Omega \chi_B \diff v(t)\right)\diff t,\quad \forall B\in (0,T)\times\Omega.
\end{align*}
\begin{lemma}
    Let $\{v_k\}_k\subset L^1(0,T;BV(\Omega))$ such that $v_k(t)\to v_0(t)$ in $L^1(\Omega)$ for almost any $t\in(0,T)$. If additionally there exists $C>0$ such that 
    \begin{align}
        \sup_k\vert Dv_k(t)\vert(\Omega)\leq C,\quad \text{for almost any }t\in(0,T),\label{sup}
    \end{align}
    then $v_0\in L^1(0,T;BV(\Omega))$ and it holds
    \begin{align}
    \mathcal L^1_{|(0,T)}\otimes Dv_k(t)\overset{*}{\rightharpoonup} \mathcal L^1_{|(0,T)}\otimes Dv_0(t),\quad \text{in }\mathcal M((0,T)\times\Omega)\quad \text{as }k\to\infty.
        \label{conv}
    \end{align}
    \label{technical}
\end{lemma}
\begin{proof}
    First, observe that, thanks to the $L^1(\Omega)$ convergence $t$-a.e. and \eqref{sup}, by \cite[Proposition 3.13]{Ambrosiofuscopallara} it holds $Dv_0(t)\in \mathcal{M}(\Omega)$ for almost any $t$, 
    $$
    Dv_k(t)\overset{*}{\rightharpoonup} Dv_0(t),\quad \text{in }\mathcal M(\Omega)\quad \text{as }k\to\infty,
    $$
    for almost any $t\in(0,T)$, and also 
    $$
    \vert Dv_0(t)\vert(\Omega) \leq \liminf_{k\to\infty}\vert Dv_k(t)\vert (\Omega)\leq C,\quad \text{for a.a. }t\in(0,T), 
    $$
    entailing $v_0\in L^1(0,T;BV(\Omega))$.
Let us consider, wlog, $h\in C_c((0,T)\times\Omega)$ such that $\vert h\vert \leq 1$, then $h(t,\cdot)\in C_c(\Omega)$ and thus it holds, for almost any $t\in(0,T)$, 
\begin{align*}
    \int_\Omega h(t,x)\diff Dv_k\to \int_\Omega h(t,x)\diff Dv_0,\quad \text{as }k\to\infty.
\end{align*}
    Moreover, observe that, by \eqref{sup}, 
    \begin{align*}
      &  \left\vert  \int_\Omega h(t,x)\diff Dv_k\right\vert \leq \sup_k\vert Dv_k(t)\vert(\Omega),\quad \text{ for a.a. }t\in(0,T),
    \end{align*}
    and $\sup_k\vert Dv_k(t)\vert(\Omega)\in L^1(0,T)$ by \eqref{sup}.
    This entails, by Lebesgue's dominated convergence Theorem, that 
\begin{align*}
    \int_0^T\int_\Omega h(t,x)\diff Dv_k\ \diff t\to \int_0^T\int_\Omega h(t,x)\diff Dv_0\ \diff t,\quad \text{as }k\to\infty,
\end{align*}
for any $h\in C_c((0,T)\times\Omega)$, which is exactly \eqref{conv}. The proof is concluded.
\end{proof}
The last lemma we propose here is the generalization for time-varying BV functions of \cite[Lemma 3.7]{CicaleseNagasePisante}, which is an anisotropic version of the Reshetnyak continuity theorem. We postpone its proof to Appendix \ref{Appendix}.
\begin{lemma}
\label{Cic}
 Let $\{v_k\}_k,v_0\in L^1(0,T;BV(\Omega))$ such that the assumptions of Lemma \ref{technical} hold. Let also $K\in C(\Omega)$ be such that $\inf_{x\in \Omega}K(x)\geq K_*>0$. If additionally $\{D v_k\}_k\subset L^1(0,T;L^1(\Omega))$ and
 \begin{align}
 \lim_{k\to \infty}\int_0^T\int_\Omega \phi^{\circ}\left(x,\sqrt{K(x)}Dv_k(t)\right)\diff x=\int_0^T\left\vert \sqrt{K(x)}Dv_0(t)\right\vert_\phi(\Omega)\diff t,
 \label{fond}
 \end{align}
 then for any function $F(t,x,p)\in C([0,T]\times \Omega\times \mathbb R^d)$ satisfying
 \begin{align*}
     F(t,x, sp)=sF(t,x, p)\quad\text{ for }(t,x)\in [0,T]\times\Omega,\ p\in \mathbb R^d,\ s\geq 0,
 \end{align*}
 and 
 \begin{align}
     F(t,x, p)=0\quad\text{ for }(t,x)\not\in K_0,\ p\in \mathbb R^d,\label{van}
 \end{align}
 with $K_0$ a fixed compact subset of $(0,T)\times\Omega$, we have 
 \begin{align}
 \lim_{k\to\infty}\int_0^T\int_\Omega F\left(t,x,\sqrt{K(x)} Dv_k\right)\diff x \diff t=\int_0^T\int_\Omega F\left(t,x,\frac{\nu_{v_0}}{\phi^{\circ}(x,\nu_{v_0})}\right)\diff \left\vert \sqrt{K(x)} Dv_0\right\vert_\phi \ \diff t,
     \label{fin}
 \end{align}
 where $\nu_{v_0}:=\frac{Dv_0}{\vert Dv_0\vert}$.
\end{lemma}
\subsection{General assumptions}
\begin{enumerate}[label=$(\mathbf{A \arabic*})$, ref = $\mathbf{A \arabic*}$]
	\item \label{ass:A} The function $A:\Omega\times \R^d\to \R$  belongs to $C^{2}(\Omega\times (\R^d\setminus\{0\}))\cap C^{1}({\Omega}\times\R^d)$, it is positive on $\Omega\times (\R^d\setminus\{0\})$, and positively two-homogeneous in the second variable: 
    \begin{align*}
        A(x,\lambda p) = \lambda^2 A(x,p)
        \quad \text{for all $x\in\Omega$, $\lambda>0$ and $p\in\R^d$}.
    \end{align*}
    This implies that $A_{p}(x,p)$ is positively 1-homogeneous in the second variable and that there exist $A_0, A_1, a_1$ with $0< A_0\le A_1$ and $a_1>0$ such that 
	\begin{align*}
	    A_0 \abs{p}^2 \le A(x,p) \le A_1  \abs{p}^2,
        \quad
        \abs{A_{p}(x,p)} \le a_1 \abs{p}
	\end{align*}
     for all $(x,p)\in\Omega\times \R^d$. Additionally, we assume that $A(x,\cdot)$ is strictly convex for any $x\in \Omega$. By homogeneity, this implies that $A(x,\cdot)$ is strongly convex~\cite[Lemma 4.3]{alfaromotion}. More precisely  there exists $C_0 >0$ such that
	\begin{align}
	    \big(A_{p}(x,p)-A_{p}(x,q)\big)\cdot(p-q) \ge C_0 |p-q|^2
	    \quad\text{for all $x\in\Omega$, \, $p\ne  q\in\R^d$}.\label{strconv}
	\end{align}
     Finally we assume
     $$
     \abs{A_{p,x}(x,p)} \le a_{2}(x)|p|
     $$
     for some $a_2(x)\in\R$.

    \item \label{ass:double_well} 
    Concerning the double well potential, that we let depend on $x$, we assume that $F$ is factorized in the form 
    $$
    F(x,u) := K(x)W(u):=K(x) u^2(1-u)^2, 
    $$
   where $K\in C^1(\Omega)$ is such that $\inf_{x\in \Omega}K(x)\geq K_*>0$.

   \end{enumerate}

\subsection{Main results}

We study existence of weak solutions, in the sense below, for the system \eqref{eq:CH}. We assume $\Omega=\mathbb T^d$, $d=2,3$.

\begin{definition}[Weak solutions to Cahn-Hilliard equation with disparate mobility]\label{def:weak_sol_degenerate}
 We say that $(u,j)$ is a weak solution to the anisotropic Cahn-Hilliard equation with disparate mobility and initial condition $u_0$ if: 
 \begin{itemize}
 \item $(u,j)$ have regularity 
\begin{align*}
 & u \in L^{\infty}(0,T; H^{1}(\Omega))\cap L^2(0,T; H^2(\Omega))\cap L^{\infty}(0,T; L^4(\Omega)), \quad \p_{t} u \in L^{2}(0,T; (H^{8/5}(\Omega))'), \\ 
 & j \in L^{2}(0,T; L^{8/5}(\Omega)), \\
 & A(x,\nabla u)\in L^{\infty}(0,T; L^1(\Omega)), \quad F(x,u)\in L^{\infty}(0,T; L^1(\Omega)),\\
 & A_{p}(x,\nabla u)\in L^{\infty}(0,T; L^{2}(\Omega)), \quad A_{x}(x,\nabla u)\in L^{\infty}(0,T; L^{1}(\Omega)), \\
 & F_{u}(x,u)\in L^{\infty}(0,T; L^{4/3}(\Omega)), \quad F_{x}(x,u)\in L^{\infty}(0,T; L^{1}(\Omega)),
\end{align*}
\item $(u,j)$ satisfy the weak formulation: for all $\zeta\in C^1_c([0,T)\times\Omega),\, \xi\in C^{2}_c((0,T)\times\Omega;\mathbb R^d)$,
\begin{equation}\label{weakA}
\begin{split}
&\int_{\Omega} u_{0}\zeta(\cdot,t) + \int_{0}^{T}\int_{\Omega}u \partial_{t}\zeta + j\cdot \nabla\zeta\diff x\diff t=0, \\
&\int_{0}^{T}\int_{\Omega}j \cdot\xi \diff x\diff t=\int_{0}^{T} \int_{\Omega}\left((\eps A(x,\nabla u)+\f{1}{\eps}F(x,u))Id-\eps\nabla u\otimes A_{p}(x,\nabla u)\right):\nabla \xi\diff x\diff t\\
&+\int_{0}^{T}\int_{\Omega}\f{1}{\eps}F_{x}(x,u)\cdot\xi\diff x\diff t+\int_{0}^{T}\int_{\Omega} \eps A_{x}(x,\nabla u)\cdot\xi\diff x\diff t\\
&-\int_{0}^{T}\int_{\Omega} (\eps\nabla u \cdot A_{p}(x,\nabla u)+\f{1}{\eps}u F_{u}(x,u))\Div\xi\diff x\diff t \\& - \int_{0}^{T}\int_{\Omega} \eps u A_{p}(x,\nabla u)\cdot \nabla\Div\xi\diff x\diff t.
\end{split}
\end{equation}
\item $(u,j)$ satisfy the energy inequality for almost any $T'\in[0,T]$
\begin{equation*}
 E[u(\cdot,T')] +\int_{0}^{T'}\int_{\Omega}\f{|j|^2}{u}\diff x\diff t\le E[u_{0}].   
\end{equation*}
\end{itemize}
\end{definition}

Here $(H^{8/5}(\Omega))'$ is the topological dual of $H^{8/5}(\Omega)$. Note that this definition is tailored for the case $\Omega$ as a $d-$dimensional torus, since in the case of general domains one has to account for contact angle conditions (see, e.g., \cite{hensel2021bv}).

\begin{remark}
 The proof could be extended to higher dimensions (more precisely when the embedding $H^{2}(\Omega)\hookrightarrow L^{p}(\Omega)$ where $p>4$ holds, that is $d\le 8$) and assuming $K\in C^{2}(\Omega).$ We keep $d=2,3$ for simplicity and physical relevancy.
\end{remark}

We now state our main result concerning existence of weak solutions under the above assumptions. Without loss of generality we assume that the initial condition is a probability measure, i.e., $u_0\in \mathcal M^+(\om)$ and $\abs{u_0}(\om)=1$. We denote this space as $\mathcal P(\om)$.
\begin{thm}[Existence of Weak Solutions]
\label{thm:weak_sol_degenerate}
Suppose $u_{0}\in \mathcal{P}(\Omega)\cap H^{1}(\Omega)$ satisfies 
$$
\Phi[u_0] + E[u_0]<+\infty
$$
where $\Phi$ is the entropy \eqref{eq:entropy}.
Assume that $A(x,p)$ and $F(x,u)$ satisfy Assumptions~\eqref{ass:A} and \eqref{ass:double_well}, respectively.  Then there exists a weak solution $(u,j)$ to \eqref{eq:CH} on $[0,T]$ in the sense of Definition~\ref{def:weak_sol_degenerate}. 
\end{thm}

In what follows, we will set $\phi^{\circ}=\sqrt A$ and let $\eps>0$. To emphasize the dependence on $\eps$, we add an index to the energy $E$, so that $E_\eps$ is defined as 
\begin{align}
E_\eps(u)=\eps\int_\Omega A(x,\nabla u)d x +\frac1\eps\int_\Omega F(x,u)d x,
    \label{Energy}
\end{align}
where $F(x,u)=K(x)W(u)$, with $K\in C^1(\Omega)$ such that $\inf_{x\in \Omega}K(x)\geq K_*>0$.

In accordance with the presentation in \cite{laux_kroemer}, we keep the same notations and introduce the notion of \textit{well-prepared} initial data: assume that the sequence of initial data $u_{\eps,0}\geq0$ is such that 
\begin{align*}
  \sup_{\eps>0}  E_\eps({u_{\eps,0}})<+\infty,\quad \f{1}{|\Omega|} \int_\Omega u_{\eps,0}\diff x=1,\quad \forall \eps>0,
\end{align*}
as well as
\begin{align}
    \label{initial}
    \begin{cases}  
    u_{\eps,0}\to \chi_{\Omega_0},\quad \text{ in }L^1(\Omega),\\
    E_\eps(u_{\eps,0})\to c_0 \mathcal P_\phi^K(\Omega_0),
    \end{cases}
\end{align}
as $\eps\to 0$, where $\Omega_0\subset \Omega$ is an open bounded subset of $\Omega$ with sufficiently smooth boundary, and $c_0:=\psi(1)$, with $\psi(s):=\int_0^s2\sqrt{W(\tau)}d\tau$. Recall that $\chi_{A}$ denotes the indicator function of the measurable set $A\subset \R^d$. The conditions above are sufficient to guarantee, for any fixed $\eps>0$, the existence of a weak solution $(u_\eps,j_\eps)$ to the anisotropic Cahn-Hilliard equation, by means of Theorem \ref{thm:weak_sol_degenerate}.
From now on we set $E_0:=\sup_{\eps>0}E_\eps(u_{\eps,0})$.

First we give a notion of solutions to the anisotropic weighted Hele–Shaw flow in the classical sense.
\begin{definition}
    \label{classicalHeleshaw}
     Let $T\in(0,\infty)$. Let $\widetilde{\Omega}:=\{\Omega(t)\}_{t\in[0,T]}$ be a family of open subsets of $\om$ with smooth boundary, such that $\widetilde{\Omega}$ evolves smoothly in time and $\om(t)$ is simply connected for any $t\in[0,T]$. Assume also that the flux $j:[0,T]\times\Omega\to \R^d$ is smooth. We say that $\widetilde{\Omega}$ and $j$ solve the anisotropic weighted Hele–Shaw equations in the classical sense if they satisfy:
\begin{align}
    \label{a1}\begin{cases}
       \Div j=0,\quad \text{ in }\Omega(t),\\
       V=-j(\cdot,t)\cdot \nu,\quad \text{ on }\partial\Omega(t),
    \end{cases}
\end{align}
and
\begin{align}
    \label{a2}
    \begin{cases}
        j(\cdot,t)=-\nabla p(\cdot,t),\quad\text{in }\Omega(t),\\
        -c_0H_\phi(t)\sqrt{K}+c_0{\nabla \sqrt{K}}\cdot n_\phi=p(\cdot,t),\quad\text{on }\partial\Omega(t).
    \end{cases}
\end{align}
\end{definition}
Here, $c_0\sqrt K$ denotes the anisotropic weighted surface tension, $\nu$ the inner normal to $\partial\Omega(t)$, $H_\phi$ denotes the anisotropic mean curvature of the free boundary $\partial \om(t)$, evolving with normal velocity $V$ (in the direction of the outer normal), and $n_\phi$ is its normal vector in the sense of \eqref{normals}. In this sharp-interface model, the flux $j$ can be viewed as a fluid velocity, and $p$ can be interpreted as pressure.  As observed in \cite{laux_kroemer}, equations \eqref{a1} state that the flow is incompressible and that the free boundary is transported by the fluid velocity. Equation \eqref{a2} are Darcy’s first law and the force balance along
the free boundary between capillary forces and pressure.

We also define the notion of weak solution to the anisotropic weighted Hele-Shaw flow: 
\begin{definition}
    Let $T\in(0,\infty). $ Let $\widetilde{\Omega}:=\{\Omega(t)\}_{t\in[0,T]}$ be a family of finite perimeter sets and let $j\in L^2(0,T;L^2(\Omega;\mathbb R^d))$. We say that the pair $(\widetilde{\Omega},j)$ is a weak solution to the anisotropic weighted Hele-Shaw flow if 
    \begin{itemize}
        \item For all $\zeta\in C^1([0,T)\times\Omega)$ we have
        \begin{align}
            \label{1}
            \int_\Omega \chi_{\Omega_0}\zeta(\cdot,t)\diff x+\int_0^T\int_\Omega \chi_{\Omega(t)}\partial_t\zeta+\chi_{\Omega(t)}j\cdot \nabla \zeta\ \diff x \diff t=0,
        \end{align}
        where $ \chi_{\Omega(t)}(t,x)=\chi_{\Omega(t)}(x)$.
        \item For all $\xi\in C^{1}((0,T)\times\Omega;\mathbb R^d)$ with $\text{div}\xi=0$ we have 
        \begin{align}
         \nonumber&\int_0^T\int_{\Omega(t)}\xi\cdot j(\cdot,t)\diff x  \diff t\\&=c_0\int_0^T\int_{\partial^*\Omega(t)}\text{tr}\left[(Id-n_\phi\otimes \nu_\phi)\nabla \xi+\left(\phi^{\circ}_x(x,\nu_\phi)+\frac{\nabla K(x)}{2K(x)}\right)\otimes \xi\right]\sqrt {K(x)}\phi^{\circ}(x,\nu) d\mathcal H^{d-1}  \diff t,\label{2}
        \end{align}
        where we recall $\phi^{\circ}=\sqrt A$ and $c_0=2\int_0^1 \sqrt {W(s)}d s$.
        \item For almost any $T'\in[0,T]$ we have 
        \begin{align}
            \label{per}
            c_0 \mathcal P_\phi^K(\Omega(T'))+\int_0^{T'}\int_{\Omega(t)}\vert j\vert^2\diff x \ \diff t\leq c_0 \mathcal P_\phi^K(\Omega_0).
        \end{align}
    \end{itemize}
    \label{heleshaw}
\end{definition}
Morever,we will see that, if $j$ and $\widetilde{\Omega}$ are
smooth and are a weak solution to the anisotropic weighted Hele–Shaw flow, then the pair $(\widetilde{\Omega},j)$ solves the anisotropic weighted Hele–Shaw
flow in the classical sense.

We can now state our convergence theorem:
\begin{thm}\label{thm:sharp_interface}
 Let $d=2,3$. Let $(u_\eps, j_\eps)$ be weak solutions to the Cahn–Hilliard equation with well prepared initial data as above.
Then there exists a subsequence $\eps_l\to 0$ and a family of finite perimeter sets $\widetilde{\Omega}:=\{\Omega(t)\}_{t\in[0,T]}\subset \Omega$
such that the following hold:
\begin{itemize}
    \item We have
\begin{align}
u_{\eps_l}\to \chi_{\widetilde{\Omega}},\quad
    \text{strongly in }L^p (0,T;\,L^q(\Omega))
    \quad
    \forall\,1 \le p < \infty,  1 \le q < 4.
    \label{compactness}
\end{align}
\item There exists $j\in L^2(0,T;L^2(\Omega;\mathbb R^d))$ such that 
\begin{align}
    \label{radon}
    j_{\eps_l}\rightharpoonup\chi_{\Omega(t)}j,
\end{align}
weakly in $L^{2}((0,T); L^{1}(\Omega;\R^d)).$
\item If in addition to \eqref{compactness} and \eqref{radon} it holds
\begin{align}
    \label{energ}
    \limsup_{l\to \infty}\int_0^TE_{\eps_l}(u_{\eps_l}(\cdot,t))\diff t\leq \int_0^T c_0\mathcal P_\phi^K(\Omega(t))\diff t,
\end{align}
then $(\widetilde{\Omega},j)$ is a weak solution to the anisotropic weighted Hele-Shaw flow in the sense of Definition \ref{heleshaw}.
\item If $j$ is sufficiently smooth (say $C^1(\overline{\Omega}$)) and $\Omega(t)$, of class $C^{2,\alpha}$, $\alpha\in(0,1]$, evolves smoothly and is simply connected for all t, then $(\widetilde{\Omega}, j)$ is also a classical solution to the anisotropic weighted Hele–Shaw flow \eqref{a1}-\eqref{a2}.
\end{itemize}
\label{conv}
\end{thm}
\begin{remark}\label{rem:sharp_interface_convexity}
    In the proof of Theorem \ref{conv} we follow two different approaches. The first makes use of a suitable anisotropic version of  Reshetnyak continuity theorem (see Lemma \ref{Cic}) as already exploited in \cite{CicaleseNagasePisante}. The second one makes use of an anisotropic tilt excess to pass to the limit as in \cite{LauxUllrich}.
\end{remark}

\section{Proof of Theorem \ref{thm:weak_sol_degenerate}: Existence of weak solutions}\label{sect:existence}

In this section, we prove the existence of weak solutions to \eqref{eq:CH}, as stated in Theorem~\ref{thm:weak_sol_degenerate}. The proof is in the spirit of \cite{lisini-ch-gradient-flow,doi:10.1142/S021919972450041X,antonio2024competing,laux_kroemer} and is structured as follows:

\begin{enumerate}
    \item We use a time-discrete approximation, specifically the Jordan--Kinderlehrer--Otto (JKO) scheme, which interprets \eqref{eq:CH} as a Wasserstein gradient flow of the energy \(E[\cdot]\) defined in \eqref{eq:energy}.
    \item We show that each step of the scheme is well-defined and construct the constant interpolation curve of these steps.
    \item We obtain uniform a priori estimates for the discrete curve \(u_{\tau}\), including bounds in \(L^\infty(0,T; H^1(\Omega))\) and \(L^2(0,T; H^2(\Omega))\), using the flow interchange lemma (see Lemma~\ref{lem:flow_interchange_lemma}).
    \item Using compactness arguments, we prove that a subsequence \(u_{\tau}\) converges to \(u\).
    \item We identify the limit \(u\) as a weak solution to \eqref{eq:CH}, as defined in Definition~\ref{def:weak_sol_degenerate}.
\end{enumerate}

\subsection{Preliminaries on the JKO scheme and optimal transport}

We build solutions to the anisotropic Cahn-Hilliard equation using the JKO scheme. This variational scheme was introduced in 1998~\cite{MR1617171} in the case of the Fokker-Planck equation. The result is based on the observation that this equation is a gradient flow in the Wasserstein metric for an energy functional. The anisotropic Cahn-Hilliard equation can also be interpreted as a gradient flow with respect to the Wasserstein metric, associated with the energy functional:
$$
\tilde{E}[u] := 
\begin{cases}
E[u], & \text{if } u \in H^1(\Omega), \\
+\infty, & \text{otherwise},
\end{cases}
$$
where \(E[u]\) is defined in~\eqref{eq:energy}. 

For a fixed time step \(\tau > 0\), we define the JKO scheme for the anisotropic Cahn-Hilliard equation as a sequence of probability measures \(\{u_\tau^n\}_n\), starting with \(u_\tau^0 = u_0\). At each step, \(u_\tau^{n+1}\) is defined by solving the minimization problem:
\begin{equation}\label{eq:preliminary_JKO}
u_\tau^{n+1} \in \argmin_{u \in \mathcal{P}(\Omega)} \left\{ \tilde{E}(u) + \frac{W_2^2(u, u_\tau^n)}{2\tau} \right\},
\end{equation}
where the  Wasserstein metric \(W_2\) is given by:
$$
W_2^2(\rho, \eta) = \inf_{T : \Omega \to \Omega} \left\{ \int_\Omega d(T(x),x)^2\, \mathrm{d}\rho  \middle|  T_\sharp\rho = \eta \right\}.
$$
Here $$
d(x,y) = \inf_{k \in \mathbb{Z}^d} |x- y +k |, \qquad x, y \in \Td,
$$
and \(T_\sharp\rho = \eta\) denotes the pushforward of \(\rho\) under the map \(T\).

This sequence defines a piecewise-constant curve \(t \mapsto u_\tau(t)\) in the space of probability measures, such that \(u_\tau(0) = u_0\) and 
\begin{equation}\label{eq:def_JKO2}
u_\tau(t) = u_\tau^{n+1}, \quad t \in (n\tau, (n+1)\tau],
\end{equation}

and the goal is to prove that $u_{\tau}\to u$ when $\tau\to 0$, where $u$ is a solution of~\eqref{eq:CH}. 

We state some results that are common tools in the JKO scheme. For a detailed proof we refer to~\cite{santambrogio2015optimal}.

Kantorovich provided a dual formulation for the squared Wasserstein distance:
$$
\frac{1}{2} W_2^2(\rho, \eta) = \sup_{\varphi(x) + \psi(y) \leq \frac{1}{2} d(x,y)^2} 
\left\{ \int_\Omega \varphi \, \mathrm{d}\rho + \int_\Omega \psi \, \mathrm{d}\eta \right\}.
$$
The optimal potential \(\varphi\) in this formulation is known as the Kantorovich potential. According to the Brenier theorem~\cite{brenier1987decomposition, brenier1991polar}, \(\varphi\) is Lipschitz continuous, and the function \(\frac{|x|^2}{2} - \varphi(x)\) is convex. Furthermore, \(\varphi\) is related to the optimal transport map \(T\) between $\rho$ and $\eta$ through \(T(x) = x - \nabla\varphi(x)\). Using the dual formulation, the first variation of the Wasserstein distance with respect to \(\rho\) is given by \(\varphi\). For the reader interested in the details, we refer to~\cite{santambrogio2015optimal}.

The optimality condition for \(u_\tau^{n+1}\) can then be written as:
$$
\frac{\varphi}{\tau} - \Div(A_{p}(x,\nabla u_\tau^{n+1} ))+ F_u(x,u_\tau^{n+1}) = C \quad \text{a.e. on } \mathrm{supp}(u_\tau^{n+1}),
$$

where \(C\) is a constant and \(\varphi\)  is the Kantorovich potential for the transport from \(u_\tau^n\) to \(u_\tau^{n+1}\). For this formulation to hold, we need \textit{a priori} to prove that for all $n$, $u_{\tau}^{n}$ is in $H^{2}(\Omega)$, but this can be justified as in~\cite{laux_kroemer, lisini-ch-gradient-flow}. In fact later, we also obtain the $H^2(\Omega)$ estimate using the previous formulation. In some sense the proof of~\cite{laux_kroemer, lisini-ch-gradient-flow} is a rigorous way to obtain the regularity that we show here more formally.

In particular, taking the gradient in the previous equation and multiplying by $u_{\tau}^{n+1}$ we obtain
\begin{equation}\label{eq:optimality_condition}
u_{\tau}^{n+1}\frac{\nabla\varphi}{\tau} - u_{\tau}^{n+1}\nabla\Div(A_{p}(x,\nabla u_\tau^{n+1} )) + u_{\tau}^{n+1}(F_{uu}(x,u_\tau^{n+1})\nabla u_{\tau}^{n+1} + F_{ux}(x,u_{\tau}^{n+1})) = 0\quad \text{a.e. in $\Omega$ },
\end{equation}

This equality is useful when we apply the flow interchange lemma stated below. 
We recall the following result, known as the flow interchange lemma:
\begin{lemma}[Flow interchange lemma]\label{lem:flow_interchange_lemma}
Let \(\rho, \eta \in \mathcal{P}(\Omega)\) be two probability measures, and let \(F\) be a convex function satisfying \(F(0) = 0\) and the McCann condition of geodesic convexity, that is $r\mapsto r^{-d}F(r^d)$ is convex and decreasing. Assume that $\rho\nabla F'(\rho)\in L^{1}(\Omega; \R^d)$. Let \(\varphi\) be the Kantorovich potential for the transport from \(\rho\) to \(\eta\). Then, the following inequality holds:
$$
\int_\Omega F(\eta) \, \mathrm{d}\eta \geq \int_\Omega F(\rho) \, \mathrm{d}\rho - \int_\Omega \rho \nabla(F'(p)) \cdot \nabla\varphi.
$$
\end{lemma}

\begin{proof}
 The proof is exactly the one from~\cite[Lemma 2.4]{dimarino}. However the authors prove it on a convex domain $\Omega$ and assume $\rho$ is Lipschitz continuous. The proof can be easily adapted on the torus, as the only point to check is an integration by parts where the boundary term vanishes. Another point to check is that geodesics curve stay inside the domain, which is the case for the torus (a similar proof would not work on open non-convex domains subsets of $\R^d$ for instance). Concerning Lipschitz continuity we can always work by approximation and obtain the final inequality stated in the lemma. Indeed $\nabla\varphi = x-T(x)$ is bounded in $L^{\infty}(\Omega)$ on the torus as both $x$ and $T(x)$ stay inside the domain. Therefore assuming $\rho\nabla F'(\rho)\in L^{1}(\Omega;\R^d)$ is enough.
\end{proof}

This lemma is particularly useful to obtain $H^2$ estimates on the solution of the JKO scheme. Indeed, it allows to compute the dissipation of another functional (in this case the entropy, which is known to be geodesically convex) along the solutions of the JKO scheme.

\subsection{JKO scheme}

Before starting the proof, we state a lemma which proves strong compactness of a sequence with bounded energy and lower semi-continuity of the energy functional with respect to the same sequence.

\begin{lemma}[Lower semi-continuity of the energy functional]\label{lem:property_E}
Let $(u_k)_{k\in\mathbb{N}}$ be a sequence of probability measures and of uniformly bounded energy, that is there exists a constant $C>0$ such that for all $k$, 
$$
\tilde{E}[u_k] \le C. 
$$
Then, there exists $u$ such that up to a subsequence (not relabeled), 
\begin{align*}
\nabla u_k &\rightharpoonup \nabla u \quad \text{weakly in $L^{2}(\Omega)^d$},\\
u_k &\to u \quad \text{strongly in $L^{p}(\Omega)$, for $1\le p<6$}.
\end{align*}
Moreover, 
$$
E[u]\le \liminf_{k\to+\infty}E[u_k]. 
$$
\end{lemma}

\begin{proof}[Proof of Lemma~\ref{lem:property_E}]

From the uniform bound $\tilde{E}[u_k] \leq C$ and the definition of $E[u]$ in \eqref{eq:energy}, we have:
\begin{equation}\label{energy_decomp}
\eps \int_\Omega A(x,\nabla u_k)\diff x + \frac{1}{\eps} \int_\Omega F(x,u_k)\diff x \leq C.
\end{equation}

Using the lower bound from Assumption~\eqref{ass:A}:
$$
A(x,p) \geq A_0|p|^2 \quad \forall x\in\Omega,\,p\in\R^d,
$$
we get from \eqref{energy_decomp}:
$$
\eps A_0 \int_\Omega |\nabla u_k|^2\diff x \leq \eps \int_\Omega A(x,\nabla u_k)\diff x \leq C.
$$
Thus:
\begin{equation}\label{grad_bound}
\|\nabla u_k\|_{L^2(\Omega)} \leq \frac{C^{1/2}}{(\eps A_0)^{1/2}} \quad \forall k\in\N.
\end{equation}

From Assumption~\eqref{ass:double_well} and the structure $F(x,u) = K(x)u^2(1-u)^2$ with $K(x) \geq K_* > 0$, we have:
$$
F(x,u) \geq K_* u^2(1-u)^2 \geq \frac{K_*}{2}u^4 \quad \text{when } u \geq 4.
$$
Using \eqref{energy_decomp}:
$$
\frac{K_*}{2\eps} \int_{u_k \geq 4} |u_k|^4\diff x \leq \frac{1}{\eps} \int_\Omega F(x,u_k)\diff x \leq C.
$$
Using the nonnegativity of $u_k$ this gives:
$$
\int_\Omega |u_k|^4\diff x \leq \frac{4\eps C}{K_*} + \int_{0\le u_k\le 4} |u_k|^4\diff x \leq \frac{4\eps C}{K_*} + 4^4|\Omega|.
$$
Hence:
\begin{equation}\label{L4_bound}
\|u_k\|_{L^4(\Omega)} \leq C \quad \forall k\in\N.
\end{equation}

From \eqref{grad_bound} and \eqref{L4_bound}, the sequence $\{u_k\}$ is bounded in $H^1(\Omega)\cap L^4(\Omega)$. By weak compactness there exists
$u \in H^1(\Omega)\cap L^4(\Omega)$ and a subsequence (not relabeled) such that:
$$
u_k \rightharpoonup u \quad \text{weakly in } H^1(\Omega),
$$
$$
u_k \rightharpoonup u \quad \text{weakly in } L^4(\Omega).
$$

For the gradient term, since $A(x,\cdot)$ is convex (Assumption~\eqref{ass:A}) and $\nabla u_k \rightharpoonup \nabla u$ weakly in $L^2(\Omega)^d$, we have by weak lower semicontinuity:
$$
\int_\Omega A(x,\nabla u)\diff x \leq \liminf_{k\to\infty} \int_\Omega A(x,\nabla u_k)\diff x.
$$
For the potential term,  $u_k \rightharpoonup u$ weakly in $L^4(\Omega)$, $F(x,u_k) = K(x)u_k^2(1-u_k)^2$  can be written as $F=F_1 + F_2$ where $F_1$ is convex and $F_2$ is uniformly bounded, $K\in C^2(\Omega)$, we obtain:
$$
\lim_{k\to\infty} \int_\Omega F(x,u_k)\diff x = \int_\Omega F(x,u)\diff x.
$$
Combining both results:
$$
E[u] \leq \liminf_{k\to\infty} E[u_k].
$$

\end{proof}

The previous lemma is particularly useful to prove that the sequence of the JKO scheme is well-defined. We also have some first estimates, classical in the JKO scheme. 
\begin{proposition}\label{prop:scheme_wellposedness}
Let $\tau>0$. Let $u_{0}\in\mathcal{P}(\Omega)$ such that $\tilde{E}[u_{0}]<+\infty$. Then the scheme defined by~\eqref{eq:preliminary_JKO} is well defined. Moreover, there exists $C$ independent of $\tau$ and $n$ such that
$$
\int_{\Omega}|\nabla u_{\tau}^{n}|^2\diff x \le C, \quad \int_{\Omega}|u_{\tau}^{n}|^{4}\diff x\le C. 
$$
This implies that the interpolation curve $u_{\tau}$ is bounded uniformly in $L^{\infty}(0,T; H^{1}(\Omega))\cap L^{\infty}(0,T; L^{4}(\Omega))$. Moreover  $u_{\tau}$ is equicontinuous in the Wasserstein metric, i.e., 
\begin{equation}\label{equicontinuity_wasserstein}
W_{2}(u_{\tau}(t),u_{\tau}(s)) \le \sqrt{2E[u_0]}\sqrt{|t-s|+\tau}.
\end{equation}
\end{proposition} 

\begin{proof}
We assume the first terms $(1,...,n)$ of the sequence to be already constructed and $E[u_{\tau}^n]<+\infty$. The functional of the scheme defining $u_{\tau}^{n+1}$ is bounded as $E(u_{\tau}^n) + \f{W^{2}_{2}(u_{\tau}^{n},u_{\tau}^{n})}{2\tau}= E(u_{\tau}^n)<+\infty$. We define a minimizing sequence $\{v_{k}\}_k$. Of course, we can assume $E[v_k]\le C$ for a uniform constant $C$, otherwise we do not reach the minimum. Since the squared Wasserstein distance is lower semi-continuous and with Lemma~\ref{lem:property_E} we conclude that $v_k$ converges to a minimizer that we call $u^{k+1}_{\tau}$. Therefore the scheme is well posed. 

The estimates on $u_{\tau}^{n}$ uniform in $n$ are a consequence of the fact that the energy of the solutions remains bounded (as $E[u_{\tau}^{n+1}]\le E[u_{\tau}^{n}]$ for instance) and the computations performed in Lemma~\ref{lem:property_E}. Concerning the estimate on the Wasserstein distance, from the optimality of $u_\tau^{n+1}$:
$$
\tilde{E}(u_\tau^{n+1}) + \frac{1}{2\tau} W_2^2(u_\tau^{n+1}, u_\tau^n) \leq \tilde{E}(u_\tau^n).
$$
Summing over $n=0,\dots,N-1$ (where $N\tau \leq T$):
$$
\tilde{E}(u_\tau^N) + \frac{1}{2\tau} \sum_{n=0}^{N-1} W_2^2(u_\tau^{n+1}, u_\tau^n) \leq \tilde{E}(u_0).
$$
Hence, $\tilde{E}(u_\tau^n) \leq \tilde{E}(u_0)$ for all $n$, and:
\begin{equation}\label{eq:telescopic_wasserstein}
\sum_{n=0}^{N-1} W_2^2(u_\tau^{n+1}, u_\tau^n) \leq 2\tau \tilde{E}(u_0).
\end{equation}

For $t,s \in [0,T]$ with $t \in (n\tau, (n+1)\tau]$ and $s \in (m\tau, (m+1)\tau]$, the triangle inequality gives:
$$
W_2(u_\tau(t), u_\tau(s)) \leq \sum_{k=m}^{n-1} W_2(u_\tau^{k+1}, u_\tau^k).
$$
Using the Cauchy-Schwarz inequality we then obtain
$$
\sum_{k=m}^{n-1} W_2(u_\tau^{k+1}, u_\tau^k) \leq \sqrt{(n - m) \sum_{k=m}^{n-1} W_2^2(u_\tau^{k+1}, u_\tau^k)} \leq \sqrt{2\tilde{E}(u_0)(|t - s| + \tau)}.
$$
Hence:
$$
W_2(u_\tau(t), u_\tau(s)) \leq \sqrt{2\tilde{E}(u_0)} \sqrt{|t - s| + \tau},
$$
concluding the proof of the proposition.
\end{proof}

From the estimates found in the previous proposition it is possible with the Arzela-Ascoli theorem to prove that $u_{\tau}\to u$ uniformly in the Wasserstein distance, and to upgrade this convergence to a weak convergence in $H^{1}(\Omega)$. However, in the weak formulation of the anistropic Cahn-Hilliard equation~\eqref{weakA}, we observe the presence of nonlinear terms in gradients, for instance $\nabla u\cdot A_{p}(x,\nabla u)$ (which is $|\nabla u|^2$ in the isotropic case). We deduce from this that the weak $H^1$ convergence is not enough, since we need strong $H^1$ convergence. To obtain the strong $H^1$ convergence, we rely on an $H^2$ estimate, found by \emph{dissipating the entropy} in the JKO scheme with the flow interchange lemma. 

\begin{proposition}[Second order estimate on the scheme]
\label{prop3.4}
Let $u_{\tau}^{n}$ be the sequence of the scheme defined in Proposition~\ref{prop:scheme_wellposedness}. Then
\begin{align}
&\nonumber\int_{\Omega}u_{\tau}^{n+1} \log u_{\tau}^{n+1} \diff x+ \tau \int_{\Omega}\Div(A_{p}(x,\nabla u_{\tau}^{n+1}))\Delta u_{\tau}^{n+1} \diff x+ \tau\int_{\Omega}F_{uu}(x,u_{\tau}^{n+1})|\nabla u_{\tau}^{n+1}|^2\diff x \\ &+ \tau\int_{\Omega}F_{ux}(x,u_{\tau}^{n+1})\cdot\nabla u_{\tau}^{n+1}\diff x \le \int_{\Omega}u_{\tau}^{n}\log u_{\tau}^{n}\diff x.\label{ineq1}
\end{align}

As a consequence there exists $C$ such that for all  $\tau$:
\begin{equation*}
\|u_{\tau}\|_{L^{2}(0,T;H^{2}(\Omega))}\le C.
\end{equation*}
\end{proposition}

\begin{remark}
Up to integrating by parts and using~\eqref{eq:optimality_condition}, the term $ \int_{\Omega}\Div(A_{p}(x,\nabla u_{\tau}^{n+1}))\Delta u_{\tau}^{n+1} \diff x$ makes sense. At the continuous level, the same computations can be performed by dissipating the entropy $\f{d}{dt}\int_{\Omega} u\log u$. As usual in the Cahn-Hilliard equation with degenerate mobility (see \cite{elliott_cahn-hilliard_1996}), this provides $L^{2}(0,T; H^{2}(\Omega))$ where estimates on the solution. 
\end{remark}

The $H^2$ estimate comes from the term $\int_{\Omega}\Div(A_{p}(x,\nabla u_{\tau}^{n+1}))\Delta u_{\tau}^{n+1} $. Indeed in the isotropic version of the Cahn-Hilliard equation, this term reads $\int_{\Omega} |\Delta u|^2\left(= \int_{\Omega}|D^2u|^2\right)$. By performing subtle integration by parts, as in~\cite{Patrik}, we can indeed prove the following lemma.

\begin{lemma}\label{lem:H^2_estimate}
Under Assumptions~\eqref{ass:A}-\eqref{ass:double_well}, there exists $C = C(C_0,a_2,d)$ such that for any $u \in H^2(\Omega)$:
\begin{equation}\label{eq:H2-est}
\int_\Omega |D^2 u|^2 \diff x \leq C\left(\int_\Omega \Div(A_p(x,\nabla u))\Delta u \diff x + \int_\Omega |\nabla u|^2 \diff x\right).
\end{equation}
\end{lemma}

\begin{proof}
The following computations are formal as $A$ is not twice differentiable in $p$ and $u$ may not have a third order derivative. However they can be made rigorous by approximating derivatives with difference quotients as in \cite{Patrik}. We denote by $\p_{i}$ the partial derivative $\partial_{x_i}$, and we adopt Einstein's summation convention for repeated indices.
\begin{align*}
\int_{\Omega} \Div(A_{p}(x,\nabla u))\Delta u \diff x&= \int_{\Omega}\p_{i}(A_{p_i}(x,\nabla u))\p_{jj}u\diff x\\
&= -\int_{\Omega}A_{p_i}(x,\nabla u) \p_{ijj}u\diff x\\
&= \int_{\Omega}A_{p_i, p_k}(x,\nabla u)\p_{kj}u\p_{ij}u \diff x+ \int_{\Omega}A_{p_{i}x_j}(x,\nabla u)\p_{ij}u \diff x.
\end{align*}
By strong convexity assumption, the first term is bounded from below by $C_{0}\sum_{j=1}^d\int_{\Omega} |\nabla \p_j u|^2$. Concerning the second term of the right-hand side, we use the assumptions on $A_{p,x}$ and we obtain that it is bounded in absolute value by 
$$
\left|\int_{\Omega}A_{p_{i}x_j}(x,\nabla u)\p_{ij}u \diff x\right| \le C\int_{\Omega} |D^2 u||\nabla u|\diff x \le \f{C_{0}}{2}\int_{\Omega}|D^2 u|^2\diff x + C \int_{\Omega}|\nabla u|^2 \diff x.
$$
for some new $C>0$. This concludes the proof of the lemma. 
\end{proof}

\begin{proof}[Proof of Proposition~\ref{prop3.4}]
Inequality \eqref{ineq1} is a consequence of the flow interchange lemma and the optimality condition. Indeed, the optimality condition, that is Equation~\eqref{eq:optimality_condition} yields $$
\f{1}{\tau}u_{\tau}^{n+1}\nabla\varphi = u_{\tau}^{n+1}\nabla\Div(A_{p}(x,\nabla u_\tau^{n+1} )) - u_{\tau}^{n+1}(F_{uu}(x,u_\tau^{n+1})\nabla u_{\tau}^{n+1} + F_{ux}(x,u_{\tau}^{n+1}))
$$ almost everywhere in $\Omega$. We apply the flow interchange lemma, i.e.,  Lemma~\ref{lem:flow_interchange_lemma}, with $\rho=u_{\tau}^{n+1}$ and $\eta = u_{\tau}^{n}$ and with the functional $F:u\mapsto \int_{\Omega} u \log u$. This functional is known to satisfy the assumptions of geodesic convexity~\cite{santambrogio2015optimal}. Combining it with the previous equation we obtain
\begin{align*}
&\int_{\Omega}u_{\tau}^{n+1} \log u_{\tau}^{n+1} \diff x+ \tau \int_{\Omega}\Div(A_{p}(x,\nabla u_{\tau}^{n+1}))\Delta u_{\tau}^{n+1} \diff x+ \tau\int_{\Omega}F_{uu}(x,u_{\tau}^{n+1})|\nabla u_{\tau}^{n+1}|^2\diff x \\ &+ \tau\int_{\Omega}F_{ux}(x,u_{\tau}^{n+1})\cdot\nabla u_{\tau}^{n+1}\diff x \le \int_{\Omega}u_{\tau}^{n}\log u_{\tau}^{n}\diff x.
\end{align*}

We now prove the uniform estimates. First we need to control the term
$$
\int_{\Omega}F_{uu}(x,u_{\tau}^{n+1}) |\nabla u_{\tau}^{n+1}|^2\diff x. 
$$
We aim to prove that there exists a constant \( C > 0 \) such that:
$$
\int_{\Omega} F_{uu}(x, u_{\tau}^{n+1}) |\nabla u_{\tau}^{n+1}|^2 \, \diff x \geq -C \int_{\Omega} |\nabla u_{\tau}^{n+1}|^2 \, \diff x.
$$
The first and second derivatives of \( F \) with respect to \( u \) are:
\begin{align*}
F_u &= \frac{\partial F}{\partial u} = 2K(x)u(1 - u)(1 - 2u), \\
F_{uu} &= \frac{\partial^2 F}{\partial u^2} = 2K(x)\left(1 - 6u + 6u^2\right).
\end{align*}

The quadratic term \( 6u^2 - 6u + 1 \) has its minimum at \( u = \frac{1}{2} \):
$$
\min_{u \in \mathbb{R}} \left(6u^2 - 6u + 1\right) = -\frac{1}{2}.
$$
Thus:
$$
 F_{uu} \geq 2K(x)\left(-\frac{1}{2}\right) = -K(x).
$$

Since \( \Omega \) is a compact flat torus and \( K(x) \) is continuous, \( K(x) \) attains its maximum \( K^{**} = \max_{x \in \Omega} K(x) \). Therefore:
$$
F_{uu}(x, u) \geq -K(x) \geq -K^{**}.
$$

We integrate and use the $H^1$ bound from Proposition~\ref{prop:scheme_wellposedness} to deduce
$$
\int_{\Omega} F_{uu}(x,u_{\tau}^{n+1}) |\nabla u_{\tau}^{n+1}|^2 \, \diff x \geq -K^{**} \int_{\Omega} |\nabla u_{\tau}^{n+1}|^2 \, \diff x \ge - K^{**}C.
$$

Then we need to control the term 
$$
\int_{\Omega}F_{ux}(x,u_{\tau}^{n+1})\cdot \nabla u_{\tau}^{n+1}\diff x . 
$$
First observe that 
\begin{align*}
F_{ux}(x,u)\cdot \nabla u&=\nabla K(x)\cdot \nabla W(u),
\end{align*}
where we recall $W(u) =u^2(1-u)^2$, and thus $\nabla W(u)=2(u(1-u)^2+u^2(u-1))\nabla u$. Using that $K\in C^1(\om)$, and thus $\sup_{x\in \om}\abs{\nabla K(x)}\leq C$, for some $C>0$, we can make the following estimate, by means of the Sobolev embedding $H^2(\om) \hookrightarrow W^{1,4}(\om)$ in dimensions 2 and 3,
\begin{align*}
& \left|\tau\int_{\Omega}F_{ux}(x,u_{\tau}^{n+1})\cdot \nabla u_{\tau}^{n+1} 
 \diff x\right|\\&\leq C\tau\sup_{x\in \om} \abs{\nabla K(x)}\int_\om(1+\abs{u_\tau^{n+1}}^3)\abs{\nabla u_{\tau}^{n+1}}\\&
 \leq C\tau(1+\norm{\nabla u_\tau^{n+1}}^2+\norm{u_\tau^{n+1}}_{L^4}^3\norm{\nabla u_\tau^{n+1}}_{L^4}) \\&
 \leq C\tau(1+\norm{u_\tau^{n+1}}_{H^2})\\&\leq C(\omega)\tau+\tau\omega\norm{D^2u_\tau^{n+1}}^2,
\end{align*}
for some $\omega>0$ to be chosen later on, where we used the estimates from Proposition~\ref{prop:scheme_wellposedness}. 



Now we turn our attention the diffusion term 
$$
\int_{\Omega}\Div(A_{p}(x,\nabla u_{\tau}^{n+1}))\Delta u_{\tau}^{n+1},
$$
but Lemma~\ref{lem:H^2_estimate} and Proposition~\ref{prop:scheme_wellposedness} yield 
$$
C_P\int_{\Omega} |D^2u_{\tau}^{n+1}|^2\diff x - C\tau\le 
\int_{\Omega}\Div(A_{p}(x,\nabla u_{\tau}^{n+1}))\Delta u_{\tau}^{n+1}\diff x,
$$
for some $C_P>0$.

In the end, plugging all the estimates above in \eqref{ineq1} and choosing $\omega=\frac{C_P}{2}>0$, there exists a constant $C>0$ such that

\begin{align*}
\int_{\Omega}u_{\tau}^{n+1} \log u_{\tau}^{n+1} + C\tau \int_{\Omega}|D^2 u_{\tau}^{n+1}|^2  \le \int_{\Omega}u_{\tau}^{n}\log u_{\tau}^{n}+ C\tau.\\
\end{align*}

We deduce the result by induction, observing that $n\tau\le T$, concluding the proof of the proposition.
\end{proof}

Before proving that the interpolation curve converges strongly, we recall here from \cite[Theorem 2.1]{MR3761096} a version of Aubin--Lions lemma useful for establishing compactness of a sequence of solutions to JKO scheme. For the proof we refer to \cite[Theorem 2]{MR2005609}.
\begin{thm}\label{thm:general_lions_aubin}
Let $(X, \|\cdot\|_{X})$ be a Banach space. We consider 
\begin{itemize}
    \item a lower semi-continuous functional $\mathcal{F}:X \to [0,+\infty]$ with relatively compact sublevels in $X$,
    \item a pseudo-distance $g:X \times X \to [0, +\infty]$, that is $g$ is lower semicontinuous and $g(\rho,\eta)=0$ for some $\rho, \eta \in X$ such that $\mathcal{F}(\rho), \mathcal{F}(\eta)<\infty$ implies $\rho = \eta$.
\end{itemize}
Let $U$ be a set of measurable functions $u: (0,T) \times X$ with $T>0$ fixed. Assume further that
\begin{equation}\label{eq:gen_lions_aubin_ass}
\sup_{u \in U} \int_0^T \mathcal{F}(u(t))\diff t < \infty, \qquad \lim_{h\to 0} \sup_{u \in U} \int_0^T g(u(t+h), u(t)) \diff t  = 0.
\end{equation}
Then, $U$ contains a sequence $\{u_n\}$ converging in measure to some $u \in X$, i.e.
$$
\left| \{t \in [0,T]:  \|u_n - u \|_{X} > \eps \} \right| \to 0 \mbox{ as } n\to \infty,\quad \forall{\eps>0}.
$$
In particular, there exists a subsequence (not relabelled) such that
$$
u_n(t) \to u(t) \mbox{ in } X \mbox{ for a.e. } t\in [0,T].
$$
\end{thm}

\begin{proposition}\label{prop:convergence_interpolation_curve}
Let $u_{\tau}$ be the constant interpolation curve. Then, up to a subsequence (not relabeled), there exists $u\in L^{\infty}(0,T;H^{1}(\Omega))\cap L^{\infty}(0,T; L^{4}(\Omega))\cap L^{2}(0,T; H^{2}(\Omega))$ such that
\begin{align*}
&u_{\tau}\rightharpoonup u\quad \text{weakly in $L^{\infty}(0,T; H^{1}(\Omega))\cap L^{\infty}(0,T; L^{4}(\Omega))\cap L^{2}(0,T; H^{2}(\Omega))$},\\
&u_{\tau}\to u\quad \text{in $L^{p}(0,T; L^{q}(\Omega))$ for all $1\le p<+\infty$ and $1\le q<6$},\\
&\nabla u_{\tau}\to \nabla u \text{ a.e. in }(0,T)\times\Omega\text{ and in } L^{2}(0,T; L^{2}(\Omega)),\\
&A_p(x,\nabla u_{\tau})\to A_p(x,\nabla u) \text{ a.e. in }(0,T)\times\Omega\text{ and in }L^{2}(0,T; L^{2}(\Omega)),\\
& A(x,\nabla u_\tau)\to A(x,\nabla u)\text{ a.e. in }(0,T)\times\Omega\text{ and in }L^{1}(0,T; L^{1}(\Omega)),\\
& A_{x}(x,\nabla u_\tau)\to A_{x}(x,\nabla u) \text{ a.e. in }(0,T)\times\Omega\text{ and in }L^{1}(0,T; L^{1}(\Omega)),\\
& F_{u}(x,u_\tau)\to F_u(x,u) \text{ a.e. in }(0,T)\times\Omega\text{ and in }L^{p}(0,T; L^{q}(\Omega)),\ \forall\text{$1\le p<+\infty$, $1\le q<2$},\\
& F(x,u_\tau)\to F(x,u) \text{ a.e. in }(0,T)\times\Omega\text{ and in }L^{p}(0,T; L^{q}(\Omega)),\ \forall 1\le p<+\infty,\ 1\le q<3/2.
\end{align*}
\end{proposition}

\begin{proof}
Propositions~\ref{prop:scheme_wellposedness} and \ref{prop3.4} show that the interpolation curve $u_{\tau}$ is bounded uniformly in $L^{\infty}(0,T; H^{1}(\Omega))\cap L^{\infty}(0,T; L^{4}(\Omega))\cap L^{2}(0,T; H^{2}(\Omega))$. This implies the weak convergence in the first line of the statement of the proposition above. We prove the strong compactness in $L^{2}(0,T;H^{1}(\Omega))$. First, we want to apply Theorem \ref{thm:general_lions_aubin} with the Banach space $X = H^1(\Omega)$, set $U = \{ u_{\tau}\}_{\tau > 0}$, pseudometric 
$
g(u_1,u_2)= W_2^2(u_1,u_2)$ (extended to $+\infty$ in case $u_1$ or $u_2$ are not probability measures) and functional $\mathcal{F}$ defined as 
$$
\mathcal{F}(u) = \begin{cases}
\|u\|^2_{H^2(\Omega)} &\mbox{ if } u \in H^2(\Omega) \cap \mathcal{P}(\Omega), \\
+ \infty &\mbox{ if } u \notin H^2(\Omega) \cap \mathcal{P}(\Omega).
\end{cases}
$$
$\mathcal{F}$ is lower semi-continuous and its level sets are compact in $H^1(\Omega)$ by the Rellich-Konchadrov theorem. Furthermore, $g$ is lower semi-continuous. Finally, \eqref{eq:gen_lions_aubin_ass} follows from the uniform estimates in $L^2(0,T; H^2(\Omega))$ and estimate \eqref{equicontinuity_wasserstein}.  \\

\noindent Therefore, Theorem~\ref{thm:general_lions_aubin} gives us a subsequence (not relabelled) such that
$$
\| u_\tau(t,\cdot) - u(t,\cdot) \|_{H^1(\Omega)}^2 \to 0 \mbox{ for a.e. } t\in [0,T].
$$
As the sequence $\{u_\tau\}_{\tau}$ is bounded in $L^{\infty}(0,T; H^1(\Omega))$, using the Lebesgue dominated convergence theorem we have proved that $u_{\tau}\to u$ strongly in $L^{p}(0,T; H^{1}(\Omega))$ for all $1\le p<+\infty$. Since by Assumption~\eqref{ass:A}
$$
|A_{p}(x,\nabla u)|\lesssim|\nabla u|, 
$$
the generalized Lebesgue dominated convergence theorem implies that $A_{p}(x,\nabla u_{\tau})\to A_{p}(x,\nabla u)$ in $L^{p}(0,T; L^{2}(\Omega))$ for all $1\le p<+\infty$. The other convergences follow in a similar manner, by interpolation and application of the generalized Lebesgue dominated convergence theorem.
\end{proof}

Before concluding that our solution is a weak solution to the desired PDE, we prove  weak compactness on the flux, which allows to provide a better weak formulation. 

\begin{proposition}[Weak compactness for the flux]\label{prop:flux_JKO}
Let $J_{\tau}^{n+1} = u_{\tau}^{n+1}\nabla\mu_{\tau}^{n+1}$ where $$
\mu_{\tau}^{n+1} = - \Div(A_{p}(x,\nabla u_\tau^{n+1} ))+ F_u(x,u_\tau^{n+1}).$$
Let $J_{\tau}$ be the constant interpolation curve of $J_{\tau}^{n+1}$. Then, up to a (not relabled) subsequence, $J_{\tau}\rightharpoonup j$ weakly in $L^{2}(0,T; L^{8/5}(\Omega))$.
\end{proposition}

\begin{proof}
The optimality condition Equation~\eqref{eq:optimality_condition} yields
$$
J_{\tau}^{n+1}=-u_{\tau}^{n+1}\f{\nabla\varphi}{\tau}.
$$
We make the observation that the Wasserstein distance between $u_{\tau}^{n+1}$ and $u_{\tau}^{n}$ can be written in terms of $J_{\tau}^{n+1}$. More precisely
\begin{align*}
W_{2}^{2}(u_{\tau}^{n}, u_{\tau}^{n+1}) = \int_{\Omega}u_{\tau}^{n+1}|x-T(x)|^2\diff x= \int_{\Omega}u_{\tau}^{n+1}|\nabla\varphi(x)|^2\diff x=\tau^2\int_{\Omega}\f{(J_{\tau}^{n+1})^2}{u_{\tau}^{n+1}}\diff x.
\end{align*}
Therefore 
$$
\left(\int_{\Omega}|J_{\tau}^{n+1}|^{8/5}\right)^{5/8}=\int_{\Omega}\left(\f{|J_{\tau}^{n+1}|}{\sqrt{u_{\tau}^{n+1}}}\right)^{8/5} \left|\sqrt{u_{\tau}^{n+1}}\right|^{8/5}\diff x \le \left(\int_{\Omega}\f{(J_{\tau}^{n+1})^2}{u_{\tau}^{n+1}}\diff x\right)^{1/2}\int_{\Omega}(u_{\tau}^{n+1})^4\diff x.
$$
We obtain 
$$
\|J_{\tau}^{n+1}\|_{L^{8/5}}\lesssim \f{W_{2}(u_{\tau}^{n},u_{\tau}^{n+1})}{\tau},
$$
and, using~\eqref{eq:telescopic_wasserstein} we obtain
$$
\|J_{\tau}\|^2_{L^{2}(0,T; L^{8/5}(\Omega))}\le \sum_{n=0}^{N}\tau \|J_{\tau}^{n+1}\|_{L^{8/5}}^2\le C.
$$
This yields the result.
\end{proof}

We can now conclude that we have found a weak solution of the equation. 

\begin{proposition}
 The function $u, j$ is a weak solution to the PDE as in Definition~\ref{def:weak_sol_degenerate}. 
\end{proposition}
\begin{proof}
The proof can be adapted from~\cite[Section 8.3]{santambrogio2015optimal}; therefore, we do not repeat the computations. It uses the convergence results obtained in Proposition~\ref{prop:convergence_interpolation_curve}, as well as the property:
$$
\Div(A_p(x,\nabla u ))\nabla u = \Div( \nabla u\otimes A_p(x,\nabla u) ) - \nabla(A(x,\nabla u))+ A_x(x,\nabla u),$$
which holds for every $u$ sufficiently regular.
\end{proof}

\section{Proof of Theorem \ref{thm:sharp_interface}: Sharp interface limit}\label{sect:sharp_interface}

We study the limit $\eps\to 0$, that is Theorem~\ref{thm:sharp_interface}. We provide two proofs of the convergence to an anisotropic weighted Hele--Shaw flow. The first is based upon an anisotropic Reshetnyak continuity theorem, as used already in~\cite{CicaleseNagasePisante}, while the second uses anisotropic tilt excess estimates extending the method of~\cite{LauxUllrich,laux_kroemer}. 

\subsection{Step 1: Compactness}

We first prove that (up to a subsequence) $u_{\eps}$ converges strongly in suitable spaces, implying the compactness in \eqref{compactness}. Our proof adapts and extends the argument presented in \cite[Section~4.1]{laux_kroemer}, which in turn builds on ideas from Modica--Mortola \cite{Modica}.

\medskip

Set 
\begin{equation}
\label{eq:definition_psi}
    \psi \colon \mathbb{R} \to \mathbb{R}, 
    \quad 
    \psi(s) = 2 \int_{0}^{s} \sqrt{W(r)} \,\diff r.
\end{equation}
Then define the functional 
\begin{equation}
\label{eq:definition_F}
    \mathcal{F} \colon L^1(\Omega) \to [0, \infty],
    \quad
    \mathcal{F}(u) = 
    \begin{cases}
      \displaystyle
      \int_\Omega  \lvert \nabla (\psi \circ u)  \rvert \,\diff x,
      & \text{if } \psi \circ u \in BV(\Omega), \\
      +\infty,
      & \text{otherwise}.
    \end{cases}
\end{equation}
We show that $\mathcal{F}$ can be controlled by $E_{\eps}$, which then yields uniform bounds leading to compactness.

\medskip

Note that, by the chain rule,
$$
     \lvert \nabla(\psi \circ u)  \rvert
    = 2\,\sqrt{W (u )}\, \lvert \nabla u  \rvert.
$$
Recalling that $A_0 \lvert p\rvert^2 \le A(x,p) \le A_1 \lvert p\rvert^2$ for all $x \in \Omega$ and $p \in \mathbb{R}^n$, that $K(x) \ge K_*>0$, and using Young's inequality, it follows that 
\begin{align}
    \nonumber \lvert \nabla (\psi\circ u) \rvert
    = 2\sqrt{W(u)}\, \lvert \nabla u \rvert
    &\le  2\,\frac{\sqrt{K(x)}}{\sqrt{K_*}\sqrt{A_0}} \,\sqrt{W(u)}\, \sqrt{A(x,\nabla u)}\\&\le \frac{1}{\sqrt{A_0}\,\sqrt{K_*}}  \left(\,
      \eps\,A\! (x,\nabla u ) 
      +
      \frac{1}{\eps}\,K(x)\,W(u)
     \right).\label{eq:bound_nabla_psi_u}
\end{align}

Thus, integrating over $\Omega$ and recalling the definition of $E_\eps(u)$, we obtain
\begin{equation}
\label{eq:F_leq_Eeps}
    \mathcal{F}(u)
    \le \frac{1}{\sqrt{A_0}\,\sqrt{K_*}} E_{\eps}(u).
\end{equation}

\medskip

From \eqref{eq:F_leq_Eeps} it follows that
\begin{equation}
\label{eq:F_bound_general}
    \mathcal{F} (u_{\eps}(\cdot,t) ) \,
    \le \frac{1}{\sqrt{A_0}\,\sqrt{K_*}} E_{\eps} (u_{\eps}(\cdot,t) ) ,
\end{equation}
for almost any $t\in(0,T)$.
Since the energy estimate ensures that
\begin{align}
    E_{\eps} (u_{\eps}(\cdot,t) )
    +
    \int_{0}^{T}\int_{\Omega}\frac{\vert j_\eps\vert^2}{u_\eps}\,\diff x\,\diff t
    \le E (u_{\eps,0} ),
\label{energyestimate1}
\end{align}
for almost any $t\in(0,T)$ we deduce that
$$
    \sup_{\eps>0}
    E_{\eps} (u_{\eps}(\cdot,t) )
    \,
     \le  \,\sup_{\eps>0} E (u_{\eps,0} )
     \le  \,E_{0},
$$
for almost any $t\in(0,T)$, where $E_0$ is a uniform bound on the initial energies $E(u_{\eps,0})$.  Combining this with \eqref{eq:F_bound_general} yields
\begin{equation}
\label{eq:F_uniform_bound}
    \sup_{\eps>0} 
    \esssup_{t\in(0,T)} \mathcal{F} (u_{\eps}(\cdot,t) ) \
     \le  C\,E_0,
\end{equation}
for some constant $C>0$ independent of $\eps$.

\medskip

We now apply a compactness argument in the spirit of the Aubin--Lions lemma, useful for the Wasserstein metric setting (see Theorem~\ref{thm:general_lions_aubin}).  Define
$$
    X =  \left\{
      u \in L^{1}(\Omega)
       : 
      \|u\|_{L^{4}(\Omega)} \le C,  
      \|\psi \circ u\|_{BV(\Omega)} \le C,  
      \f{1}{|\Omega|}\int_{\Omega} u \,\diff x = 1
     \right\}
$$
for some $C>0$, and consider $g$ to be the $2$-Wasserstein distance on the space of probability densities.  The functional $\mathcal{F}$ in \eqref{eq:definition_F} has relatively compact sublevels in $X$: 

\begin{enumerate}
\item[(i)] If a sequence $\{u_k\}_k$ is bounded in $X$, then $\|\psi \circ u_k\|_{BV(\Omega)} \le C$ implies that $\{\psi \circ u_k\}_k$ is precompact in $L^1(\Omega)$ and thus convergent (up to a subsequence) almost everywhere.  
\item[(ii)] Since $\psi$ is invertible, this gives $u_k \to u$ pointwise a.e.\ for some limit $u$.  
\item[(iii)] The uniform $L^4$ bound of $\{u_k\}_k$ ensures equiintegrability, and Vitali's theorem then improves the convergence to strong $L^1(\Omega)$ convergence of $u_k$ to $u$.
\end{enumerate}

\medskip

To use Theorem~\ref{thm:general_lions_aubin}, we also need an equicontinuity in time with respect to the Wasserstein distance. From \eqref{equicontinuity_wasserstein} which passes to the limit when $\tau\to 0$ and lower semi-continuity of the squared Wasserstein distance (or by reproving this bound directly on the weak solutions of the equation), we have for all $\eps>0$
$$
    W_{2} (u_{\eps}(t),\,u_{\eps}(s) )
    \le C\,\sqrt{t-s},
    \qquad
    \forall\,0 \le s \le t \le T,
$$
where $C>0$ does not depend on $\eps$. 
By applying Theorem~\ref{thm:general_lions_aubin}, we deduce that there exists a subsequence $\eps_{l}$ such that
$$
    u_{\eps_{l}}(t)
    \to
    u_{\ast}(t)
    \quad\text{in}\quad
    L^{1}(\Omega)
    \quad
    \text{for a.e.}\ t\in[0,T].
$$
Owing to the additional uniform $L^{\infty}(0,T; L^{4}(\Omega))$-bounds, by Lebesgue dominated convergence theorem we deduce
$$
    u_{\eps_{l}}
    \to
    u_{\ast} 
    \quad
    \text{strongly in } L^p (0,T;\,L^q(\Omega))
    \quad
    \forall\,1 \le p < \infty,  1 \le q < 4.
$$
This completes the proof of the compactness statement in \eqref{compactness}.

With the strong convergence \eqref{compactness} at hand, we now aim at identifying the quantity $u_*$. The proof that $u_*(t,x)\in \{0,1\}$ for almost any $(t,x)\in [0,T]\times\Omega$ can immediately be deduced from the fact that, by Fatou's Lemma 
\begin{align*}
   & \int_\Omega W(u_*(\cdot,t))\diff t=\int_\Omega \liminf_{l\to \infty} W(u_{\eps_l}(\cdot,t))\diff x\\&\leq \frac{1}{K_*}\liminf_{l\to \infty}\int_\Omega K(x)W(u_{\eps_l}(t,x))\diff x\leq \liminf_{l \to \infty}\frac{\eps_l}{K_*} E_0=0, 
\end{align*}
entailing the result since $W$ vanishes only at $0,1$. Therefore, for almost any $(t,x)\in [0,T]\times\Omega$, we have $u(t,x)\in\{s\in \mathbb R:\ W(s)=0\}=\{0,1\}$. We now set $\Omega(t):=\{x\in \Omega:\ \lim_{\eps_l\to 0}u_{\eps_l}(t,x)=1\}$, and show that this set is of finite perimeter. Clearly, we will then have $u_{\eps_l}\to \chi_{\Omega(t)}$ in $L^1(0,T;L^1(\Omega))$. In order to show that $\Omega(t)$ is of finite perimeter, we have to operate as follows. First, by Fatou's Lemma, for any $\xi\in C^1_c(\Omega;\mathbb R^d)$ with $\vert \xi\vert\leq 1$, 
\begin{align*}
    &\int_\Omega c_0\chi_{\Omega(t)}\text{div}\xi\leq \liminf_{l\to \infty}\int_\Omega (\psi\circ u_{\eps_l})(t,x)\text{div}\xi\\&
    \leq \liminf_{l\to\infty}\int_\Omega \vert \nabla(\psi\circ u_{\eps_l})\vert(t,x)\diff x\leq \frac{C_2}{\sqrt {K_*}}\liminf_{l\to\infty}E_{\eps_l}(u_{\eps_l})<+\infty, 
\end{align*}
with $c_0=\psi(1)$ and where in the last step we operated as in \eqref{eq:bound_nabla_psi_u}, recalling the properties of $A$. By taking the supremum over the functions $\xi$ we indeed obtain by definition that $\mathcal P_\om(\Omega(t))<+\infty$ for almost any $t\in(0,T)$.

We now address the convergence of the flux \(j_{\eps_l}\) and show that it belongs to \(L^2 ((0,T)\times \Omega(t) )\).  
Recall that by definition and the energy estimate, we have
$$
    \int_{0}^{T} \int_{\Omega} \frac{\lvert j_{\eps}\rvert^2}{u_{\eps}} \,\diff x\,\diff t 
      \le   E (u_{\eps,0} ),
$$
uniformly in \(\eps\). 

\medskip

\noindent
We exploit the factorization
$$
    j_{\eps_l}
    =
    \frac{j_{\eps_l}}{\sqrt{u_{\eps_l}}}\,\sqrt{u_{\eps_l}}.
$$
Since 
\begin{align}
     \norm{ \f{j_{\eps_l}}{\sqrt{u_{\eps_l}}} }_{L^2((0,T)\times\Omega)}
      \le  
    C
\label{uniformbound}
\end{align}
by the above energy estimate, there exists (up to a subsequence, not relabeled) some $j_{\ast} \in L^2 ((0,T)\times\Omega )$ such that
$$
    \frac{j_{\eps_l}}{\sqrt{u_{\eps_l}}}
    \rightharpoonup
    j_{\ast}
    \quad
    \text{weakly in } L^2 ((0,T)\times\Omega ).
$$
On the other hand, we already know that \(u_{\eps_l}\to \chi_{\Omega(t)}\) strongly in \(L^1 ((0,T)\times\Omega )\), and hence 
\(\sqrt{u_{\eps_l}} \to \chi_{\Omega(t)}\) strongly in \(L^2 ((0,T)\times\Omega )\). Therefore, we may pass to the limit in the product:
$$
    j_{\eps_l}
    =
     \left(\frac{j_{\eps_l}}{\sqrt{u_{\eps_l}}} \right)\,\sqrt{u_{\eps_l}}
    \rightharpoonup
    j_{\ast}\,\chi_{\Omega(t)}
    \quad
    \text{weakly in } L^1 ((0,T)\times\Omega ).
$$
Define
$$
    j := j_{\ast}\,\chi_{\Omega(t)},
$$
which means that $\text{supp } j\subset \overline{\bigcup_{t\in[0,T]}\om(t)\times\{t\}}$.
Then we obtain
$$
    j_{\eps_l}
    \rightharpoonup
    j
    \quad
    \text{weakly in } L^1 ((0,T)\times\Omega ).
$$

\medskip

We claim that \(j\in L^2 ((0,T)\times \Omega )\). To see this, we apply a variant of Ioffe's theorem (see, e.g., \cite{Ambrosiofuscopallara}) that can be adapted to the space time settings. In a simplified form suitable for our problem, it states:

\begin{thm}[Ioffe]
\label{thm:ioffe}
    Let \(f:\,U\times \mathbb{R}^{m+k}\to [0,+\infty]\) be a normal function satisfying: for each fixed \(x \in U\) and \(s \in \mathbb{R}^m\), the map \(z \mapsto f(x,s,z)\) is convex in \(\mathbb{R}^k\).
    Assume that
    $$
        u_h \to u \quad \text{strongly in } L^1(U)^m,
        \quad
        v_h \rightharpoonup v \quad \text{weakly in } L^1(U)^k.
    $$
    Then
    $$
        \liminf_{h\to\infty}\int_{U} f (x,u_h,v_h )\,\diff x
        \ge\
        \int_{U} f(x,u,v)\,\diff x.
    $$
\end{thm}

\noindent
In our settings we choose $U=(0,T)\times\Omega$ and 
$$
    f(x,s,z)  = 
    \begin{cases}
      \displaystyle
      \frac{\lvert z\rvert^2}{s}, 
      & s>0,
      \\
      0, 
      & s\le 0,
    \end{cases}
$$
and notice that for each fixed \((x,s)\) with \(s>0\), the map \(z\mapsto |z|^2/s\) is clearly convex in \(z\).  Furthermore, we have 
$$
    u_{\eps_l}  \to  u\quad \text{strongly in }L^1(U), 
    \quad
    j_{\eps_l}  \rightharpoonup  j \quad \text{weakly in }L^1(U).
$$
We then deduce
\begin{align}
\label{eq:limit_inferior_j}
    \int_{0}^{T} \int_{\Omega(t)} |j|^2 \,\diff x\,\diff t
    \le
    \liminf_{l\to\infty} 
    \int_{0}^{T}\int_{\Omega} \frac{\lvert j_{\eps_l}\rvert^2}{u_{\eps_l}} \,\diff x\,\diff t.
\end{align}
Since the right-hand side is uniformly bounded by the initial energy (cf. \eqref{uniformbound}), it follows that
$$
    j \in L^2 ((0,T)\times \Omega ).
$$
\subsection{Step 2. Convergence}
We first show that \eqref{1} holds, by passing to the limit in the corresponding equation for $u_{\eps_l}$. Let then $\zeta\in C_c^1([0,T)\times\Omega)$. We have
\begin{align*}
    \int_\Omega u_{\eps_l,0}\zeta(\cdot,0)\diff x+\int_0^T\int_\Omega( u_{\eps_l}\partial_t\zeta+j_{\eps_l}\cdot\nabla\zeta)\diff x\diff t=0.
\end{align*}
Then the first term converges by the assumption \eqref{initial}, whereas the second one and the third converge by \eqref{compactness} and \eqref{radon}, respectively. This means that \eqref{1} holds for the limit.

The main issue is now to prove that \eqref{2} holds, i.e., for all $\xi\in C^1_c((0,T)\times\Omega)$ with $\text{div}\xi=0$ we have
\begin{align}
            &\nonumber\int_0^T\int_{\Omega(t)}\xi\cdot j(\cdot,t)\diff x \ \diff t\\&=c_0\int_0^T\int_{\partial^*\Omega(t)}tr\left[(Id-n_\phi\otimes \nu_\phi)\nabla \xi+\left(\phi^{\circ}_x(x,\nu_\phi)+\frac{\nabla K(x)}{2K(x)}\right)\otimes \xi\right]\sqrt{K(x)}\phi^{\circ}(x,\nu)d\mathcal H^{d-1} \ \diff t.
             \label{2bis}
        \end{align}
        To obtain this, we first recall the definition of $\mathcal P_\phi^K(\Omega(t))=\int_\Omega \sqrt{K(x)}\phi^{\circ}(x,\nu(x)) d\abs{D\chi_{\om(t)}}$, so that, by the proof of \cite[Proposition 4.1]{Bouchitte_Gamma} (observing that we are in a case analogous to the one in \cite[(3.22)]{Bouchitte_Gamma}), we infer 
\begin{align}
c_0\mathcal P_\phi^K(\Omega(t))\leq \liminf_{l\to\infty}\int_\Omega 2\sqrt{A(x,\nabla u_{\eps_l})K(x)W(u_{\eps_l})}\diff x , 
    \label{controlper0}
\end{align}
for almost any $t\in(0,T)$. Then, by Young's inequality, 

\begin{equation}   \label{nab2}
2\sqrt {K(x)W(u_{\eps_l}) A(x,\nabla u_{\eps_l})}\leq \eps A(x,\nabla u_{\eps_l})+\frac 1{\eps_l} K(x)W(u_{\eps_l}), 
\end{equation}  

for almost any $(t,x)\in(0,T)\times\Omega$. Therefore, we infer from \eqref{controlper0} that
\begin{align}
c_0\mathcal P_\phi^K(\Omega(t))\leq \liminf_{l\to\infty}E_{\eps
_l}(u_{\eps_l}(\cdot,t)), 
    \label{controlper}
\end{align}

for almost any $t\in(0,T)$. Integrating over $(0,T)$, this entails, by Fatou's Lemma, 
\begin{align*}
    c_0\int_0^T\mathcal{P}^K_\phi(\Omega(t))\diff t\leq \liminf_{l\to \infty}\int_0^T E_{\eps_l}(u_{\eps_l}(\cdot,t))\diff t,
\end{align*}
which, together with assumption \eqref{energ}, gives 
\begin{align}
\lim_{l\to \infty}\int_0^T E_{\eps_l}(u_{\eps_l}(\cdot,t))\diff t= c_0\int_0^T\mathcal{P}^K_\phi(\Omega(t))\diff t.
    \label{energ_covnergence}
\end{align}

This convergence, thanks to Lemma \ref{sup}, allows us prove equipartition of the energy, as observed in~\cite[Theorem 3.4]{LauxUllrich} in the context of the Allen-Cahn equation.
\begin{lemma}
\label{base}
    Under the assumptions of Theorem \ref{conv}, including the limsup energy estimate \eqref{energ}, it holds, as $l\to \infty$, 
    \begin{enumerate} [label=$(\mathbf{B \arabic*})$, ref = $\mathbf{B \arabic*}$]
    \item \label{B1}\begin{align}
 \mathcal L^1_{|(0,T)}\otimes \nabla(\psi\circ u_{\eps_l})(t)\overset{*}{\rightharpoonup}\mathcal L^1_{|(0,T)}\otimes  c_0 D\chi_{\widetilde{\Omega}}(t),\quad \text{ in }\mathcal M((0,T)\times\Omega),\label{convA}
\end{align}
and 
\begin{align}
\lim_{l\to \infty}\int_0^T\int_\Omega \phi^{\circ}(x,\sqrt {K(x)}\nabla(\psi\circ u_{\eps_l}))\diff x\diff t=c_0\int_0^T\mathcal P_\phi^K(\Omega(t))\diff t.
    \label{toshow}
\end{align}
        \item \label{B2}$\sqrt{\eps_l A(x,\nabla u_{\eps_l})}-\sqrt{\frac{1}{\eps_l}K(x)W(u_{\eps_l})}\to 0$ in $L^2((0,T)\times\Omega)$,
        \item\label{B3} $\eps_lA(x,\nabla u_{\eps_l})-\frac1{\eps_l}K(x)W(u_{\eps_l})\to 0$ in $L^1((0,T)\times\Omega)$.
        \end{enumerate}
\end{lemma}

\begin{proof}
Let us start from \eqref{B1}. First observe that, by \eqref{eq:bound_nabla_psi_u}, we immediately have
\begin{align}
\sup_\eps \int_\Omega \vert \nabla (\psi\circ u_\eps)(t)\vert \diff x \leq CE_0,
    \label{bound}
\end{align}
for almost any $t\in(0,T)$. Moreover, observe that, for $M>0$ sufficiently large it holds, by the definition of $W$ and $K\geq K_*>0$, 
\begin{align*}
    \int_{\{u_\eps\geq M\}}\vert u_\eps\vert^4 \diff x\leq \frac{C}{K_*}\int_\Omega K(x)W(u_\eps)\diff x\leq \frac{C}{K_*}E_0\eps,
\end{align*}
from the definition of $E_\eps(u_\eps)$. Then, it holds, exploiting the structure of $W$, for almost any $t\in(0,T)$, for some $p\in(1,\frac43]$,
\begin{align*}
   & \int_\Omega \vert \psi\circ u_\eps\vert^{p} \diff x \leq  \int_{u_\eps\geq M} \left(C \vert u_\eps\vert^{4}+C\vert u_\eps\vert ^3\right) \diff x +\int_{u_\eps\leq M}\vert \psi\circ u_\eps\vert^p \diff x\\&
    \leq C_M\int_{u_\eps\geq M}\vert u_\eps\vert ^4\diff x+Lip_{\psi|_{[0,M]}}^p\int_{u_\eps\leq M} \vert u_\eps\vert^p \diff x\leq C_M,\quad     \forall \eps>0.
\end{align*}
Recalling that $u_{\eps_l}\to \chi_{\Omega(t)}$ almost everywhere in $(0,T)\times\Omega$, up to subsequences, since $\psi$ is Lipschitz continuous, we have $\psi\circ u_{\eps_l}\to c_0\chi_{\Omega(t)}$ almost everywhere in $(0,T)\times\Omega$, and thus, by generalized Lebesgue's dominated convergence $\psi\circ u_{\eps_l}(t)\to c_0\chi_{\Omega(t)}\quad \text{in }L^1((0,T)\times\Omega)$, entailing, up to another subsequence, $$\psi\circ u_{\eps_l}(t)\to c_0\chi_{\Omega(t)}\text{ in }L^1(\Omega),\quad \text{for almost any }t\in (0,T).$$ 
Observe that, as already shown, $c_0\chi_{\widetilde{\Omega}}\in L^1(0,T;BV(\Omega))$.
Recalling \eqref{bound}, we can thus apply Lemma \ref{technical} with $\{v_k\}_k=\{\psi\circ u_{\eps_l}\}_l$ to deduce that 
\begin{align*}
   \mathcal L^1_{|(0,T)}\otimes D(\psi\circ u_{\eps_l})(t)=\mathcal L^1_{|(0,T)}\otimes \nabla(\psi\circ u_{\eps_l})(t)\overset{*}{\rightharpoonup}\mathcal L^1_{|(0,T)}\otimes  c_0 D\chi_{\widetilde{\Omega}}(t),\quad \text{ in }\mathcal M((0,T)\times\Omega),
\end{align*}
proving \eqref{convA}. We need now to prove \eqref{toshow}.
First, concerning the liminf part, integrating in time \eqref{controlper0} we  deduce 
\begin{align*}
    c_0\int_0^T\mathcal P_\phi^K(\Omega(t))\diff t &\leq
   \liminf_{l\to\infty}\int_{0}^{T}\int_\Omega 2\sqrt{A(x,\nabla u_{\eps_l})K(x)W(u_{\eps_l})}\diff x\diff t \\
    &=\liminf_{l\to\infty}\int_0^T\int_\Omega \phi^{\circ}(x,\sqrt{K}(x)\nabla(\psi\circ u_{\eps_l}))\diff x\diff t, 
\end{align*}
where we used in the last line the 1-homogeneity of $\phi^{\circ}$ with respect to the second variable and the definition of $\psi$. Concerning the limsup inequality, this comes directly from \eqref{energ_covnergence}. Indeed, by \eqref{nab2}, it holds
\begin{align*}
    \limsup_{l\to\infty}\int_0^T\int_\Omega \phi^{\circ}(x,\nabla(\psi\circ u_{\eps_l}))\diff x\diff t\leq \lim_{l\to\infty}\int_0^TE_{\eps_l}(u_{\eps_l})\diff t=c_0\int_0^T\mathcal P_\phi(\Omega(t))\diff t,
\end{align*}
so that \eqref{toshow} follows. 
Now, concerning \eqref{B2}, we have 
\begin{align*}
&\int_0^T\int_\Omega  \left(\sqrt{\eps_l A(x,\nabla u_{\eps_l})}-\sqrt{\frac{1}{\eps_l}K(x)W(u_{\eps_l})}\right)^2\diff x \diff t\\&
=\int_0^T E_{\eps_l}(u_{\eps_l})d t -2\int_0^T\int_\Omega  \sqrt{ A(x,\nabla u_{\eps_l})K(x)W(u_{\eps_l})}\diff x \diff t \to 0,
\end{align*}
recalling \eqref{energ_covnergence} and \eqref{toshow}. In conclusion, \eqref{B3} is a consequence of 
\begin{align*}
    &\int_0^T\int_\Omega  \left({\eps_l A(x,\nabla u_{\eps_l}})-{\frac{1}{\eps_l}K(x)W(u_{\eps_l})}\right)\diff x \diff t\\&=\int_0^T\int_\Omega  \left(\sqrt{\eps_l A(x,\nabla u_{\eps_l})}-\sqrt{\frac{1}{\eps_l}K(x)W(u_{\eps_l})}\right)\left(\sqrt{\eps_l A(x,\nabla u_{\eps_l})}+\sqrt{\frac{1}{\eps_l}K(x)W(u_{\eps_l})}\right)\diff x \diff t\\&
    \leq C(T)\left(\int_0^T\int_\Omega  \left(\sqrt{\eps_l A(x,\nabla u_{\eps_l})}-\sqrt{\frac{1}{\eps_l}K(x)W(u_{\eps_l})}\right)^2\diff x\diff t\right)^\frac12\to 0,
\end{align*}
by \eqref{B2}, where we have used the fact that 
\begin{align*}
   &\int_0^T\int_\Omega  \left(\sqrt{\eps_l A(x,\nabla u_{\eps_l})}+\sqrt{\frac{1}{\eps_l}K(x)W(u_{\eps_l})}\right)^2\diff x\diff t\\&
   = \int_0^T E_{\eps_l}(u_{\eps_l}) \diff t +2\int_0^T\int_\Omega  \sqrt{ A(x,\nabla u_{\eps_l})K(x)W(u_{\eps_l})}\diff x \diff t\leq C(T),
\end{align*}
recalling \eqref{energ_covnergence} and \eqref{toshow}. The proof of Lemma \ref{base} is concluded.
\end{proof}

Note that, from Lemma \ref{base}, we also deduce that, since $A(x,p)$ is positive 2-homogeneous in $p$, 
\begin{align}
\nonumber&\frac12\sqrt{A(x,\sqrt{K}(x)\nabla (\psi\circ u_{\eps_l}))}-\eps_lA(x,\nabla u_{\eps_l})=
\\&\sqrt{\eps_l A(x,\nabla u_{\eps_l})}\left(\sqrt{\frac1{\eps_l}K(x)W(u_{\eps_l})}-\sqrt{\eps_l A(x,\nabla u_{\eps_l})}\right)\to 0\quad \text{in }L^1((0,T)\times\Omega),
    \label{essentialsquareroot}
\end{align}
as $l\to\infty$.
Indeed, we have, by \eqref{B2},
\begin{align*}
   & \left\Vert \sqrt{\eps_l A(x,\nabla u_{\eps_l})}\left(\sqrt{\frac1{\eps_l}K(x)W(u_{\eps_l})}-\sqrt{\eps_l A(x,\nabla u_{\eps_l})}\right)\right\Vert_{L^1((0,T)\times\Omega)}\\&\leq \norm{\eps_l A(x,\nabla u_{\eps_l})}_{L^1((0,T)\times\Omega)}^\frac12 \norm{\sqrt{\frac1{\eps_l}K(x)W(u_{\eps_l})}-\sqrt{\eps_l A(x,\nabla u_{\eps_l})}}_{L^2((0,T)\times\Omega)}\\&\leq T^\frac12E_0^\frac12\norm{\sqrt{\frac1{\eps_l}K(x)W(u_{\eps_l})}-\sqrt{\eps_l A(x,\nabla u_{\eps_l})}}_{L^2((0,T)\times\Omega)}   \to 0.
\end{align*}
Analogously, recalling \eqref{B3}, we also deduce 
\begin{align}
&\frac12\sqrt{A(x,\sqrt{K}(x)\nabla (\psi\circ u_{\eps_l}))}- \frac1 {\eps_l}K(x)W(u_{\eps_l})\to 0\quad \text{in }L^1((0,T)\times\Omega),
    \label{essentialsquareroot1}
\end{align}
as $l\to\infty$. Indeed, we have by the equipartition of energy \eqref{B3} and \eqref{essentialsquareroot},
\begin{align*}
   & \left\Vert \frac12\sqrt{A(x,\sqrt{K}(x)\nabla (\psi\circ u_{\eps_l}))}- \frac1 {\eps_l}K(x)W(u_{\eps_l})\right\Vert_{L^1((0,T)\times\Omega)}\\&\leq 
   \left\Vert \frac12\sqrt{A(x,\sqrt{K}(x)\nabla (\psi\circ u_{\eps_l}))}- \eps_lA(x,\nabla u_{\eps_l})\right\Vert_{L^1((0,T)\times\Omega)}\\&+\left\Vert \eps_lA(x,\nabla u_{\eps_l})- \frac1 {\eps_l}K(x)W(u_{\eps_l})\right\Vert_{L^1((0,T)\times\Omega)}\to 0.
\end{align*}
We can now pass to consider the validity of \eqref{2bis}, starting from the equation for the flux $j_\eps$: indeed, for $\xi\in C^{1}_c((0,T)\times\Omega;\mathbb R^d)$ such that $\text{div}\xi=0$, we obtain from \eqref{weakA}. 
\begin{align}
  \nonumber &\int_0^T \int_\Omega j_{\eps}\cdot \xi\ \diff x\\
&=\int_0^T\int_\Omega \mathbf T_\eps:\nabla \xi \ \diff x\diff t+  \int_0^T\int_\Omega \eps A_x(x,\nabla u_\eps)\cdot \xi \ \diff x\diff t+\int_0^T\int_\Omega \frac1\eps W(u_\eps)\nabla K(x)\cdot \xi \ \diff x\diff t,\label{jeps}
\end{align}
where we set, as in \eqref{stresstensor}, $\mathbf T_\eps= \left(\eps A(x,\nabla u_\eps)+\frac 1 {\eps}K(x)W(u_\eps)\right)Id-\eps\left(\nabla u_\eps \otimes  A_p(x,\nabla u_\eps)\right)$, recalling that, given $B,C\in \mathbb R^{d\times d}$, $B:C:=tr(B^TC)$. 
\subsubsection{Method I. Reshetnyak-type argument}
In the first approach, in order to pass to the limit we aim at following the idea of 
\cite{CicaleseNagasePisante}, exploiting a suitable adaptation of Reshetnyak continuity theorem. First notice that we have, along the sequence $\{\eps_l\}_l$,
\begin{align*}
   & \int_0^T\int_\Omega \mathbf T_{\eps_l}:\nabla \xi+\int_0^T\int_\Omega \eps_l A_x(x,\nabla u_{\eps_l})\cdot \xi \ \diff x\diff t+\int_0^T\int_\Omega \frac1{\eps_l} K(x)W(u_{\eps_l})\nabla K(x)\cdot \xi \ \diff x\diff t\\&
   =\int_0^T\int_\Omega \left(\frac1{\epsl}K(x)W(u_{\epsl})-\epsl A(x,\nabla u_{\epsl})\right)Id:\nabla \xi \diff x\diff t\\&+\int_0^T\int_\Omega \left(2\epsl A(x,\nabla u_{\epsl})Id- \epsl\left(\nabla u_{\epsl} \otimes  A_p(x,\nabla u_{\epsl})\right)\right):\nabla \xi \diff x\diff t\\&
   +   \int_0^T\int_\Omega \eps_l A_x(x,\nabla u_{\eps_l})\cdot \xi \ \diff x\diff t+\int_0^T\int_\Omega \frac1{\eps_l} W(u_{\eps_l})\nabla K(x)\cdot \xi \ \diff x\diff t\\&
   =\underbrace{\int_0^T\int_\Omega \left(\frac1{\epsl}K(x)W(u_{\epsl})-\epsl A(x,\nabla u_{\epsl})\right)Id:\nabla \xi \diff x\diff t}_{I_1}
   \\&+\underbrace{\int_0^T\int_\Omega \left(\epsl A(x,\nabla u_{\epsl})-\frac12\sqrt{A(x,\sqrt{K}(x)\nabla(\psi\circ u_{\epsl}))}\right)
\widetilde H_{\epsl}\diff x\diff t}_{I_2}\\&
+\underbrace{\int_\Omega \left(\frac1{\eps_l} K(x) W(u_{\eps_l})\frac{\nabla K(x)}{ K(x)}-\frac12\sqrt{A(x,\sqrt{K}(x)\nabla(\psi\circ u_{\epsl}))}\frac{\nabla K(x)}{ K(x)} \right)\cdot \xi\ \diff x\diff t}_{I_3}\\&
+\underbrace{\int_0^T\int_\Omega \frac12\sqrt{A(x,\sqrt{K}(x)\nabla(\psi\circ u_{\epsl}))}H_{\epsl}\diff x\diff t}_{I_4},
\end{align*}
where 
\begin{align*}
    H_{\epsl}:=\underbrace{\left(2Id-\frac{\nabla u_{\epsl}}{\sqrt{A(x,\nabla u_{\epsl})}}\otimes  \frac{A_p(x,\nabla u_{\epsl})}{\sqrt{A(x,\nabla u_{\epsl})}}\right):\nabla \xi+\frac{ A_x(x,\nabla u_{\eps_l})}{A(x,\nabla u_{\epsl})}\cdot \xi}_{\widetilde H_{\eps_l}}+\frac{\nabla K(x)}{K(x)}\cdot \xi,
\end{align*}
and we implicitly assumed $ H_{\eps_l}=0$ on $\{x\in \Omega:\ \nabla u_{\eps_l}(x)=0\}$.
It can be seen that $H_\eps$ (and $\widetilde{H}_{\eps}$ as well)  is positively 0-homogeneous with respect to the vector corresponding to $\nabla u_{\eps_l}$, so that, since $\xi\in C^1_c((0,T)\times\Omega;\mathbb R^d)$, 
$$
\sup_{(t,x)\in(0,T)\times\Omega}\vert H_{\epsl}\vert \leq C, 
$$
uniformly in $l$. Moreover, on the set $B_{\epsl}:=\{(t,x)\in (0,T)\times\Omega: \ \vert \nabla(\psi\circ u_{\epsl})\vert_{\phi}\not=0\}$ we have $ \f{\nabla u_{\epsl}}{\vert \nabla u_{\epsl}\vert_\phi}= \f{\nabla (\psi\circ u_{\epsl})}{\vert \nabla (\psi\circ u_{\epsl})\vert_\phi}$, where we recall that, by 1-homogeneity and definition, $$\vert \nabla  u_{\epsl}\vert_\phi=\phi^{\circ}\left(x,\frac{\nabla u_{\epsl}}{\vert \nabla u_{\epsl}\vert}\right)\vert \nabla u_{\epsl}\vert =\sqrt{A(x,\nabla u_{{\epsl}})},$$
$$\vert \nabla  (\psi\circ u_{\epsl})\vert_\phi=\phi^{\circ}\left(x,\frac{\nabla (\psi\circ u_{\epsl})}{\vert \nabla (\psi\circ u_{\epsl})\vert}\right)\vert \nabla (\psi\circ u_{\epsl})\vert =\sqrt{A(x,\nabla (\psi\circ u_{\epsl}))}.$$
Therefore, we can also write
\begin{align*}
    H_{\epsl}&=\left(2Id-\frac{\nabla (\psi\circ u_{\epsl})}{\vert \nabla(\psi\circ u_{\epsl})\vert_\phi}\otimes  {A_p\left(x,\frac{\nabla (\psi\circ u_{\epsl})}{\vert 
    \nabla(\psi\circ u_{\epsl})\vert_\phi}\right)}\right):\nabla \xi\\&+{ A_x\left(x,\frac{\nabla (\psi\circ u_{\epsl})}{\vert  \nabla(\psi\circ u_{\epsl})\vert_\phi}\right)}\cdot \xi+\frac{\nabla K(x)}{K(x)}\cdot \xi.
\end{align*}
Now, since $\xi\in C^1_c((0,T)\times\Omega;\mathbb R^d)$ and $H_{\epsl},\widetilde{H}_{\epsl}$ are bounded, we easily see from Lemma \ref{base} and \eqref{essentialsquareroot} that $I_1\to 0$ and $I_2\to 0$ as $l\to \infty$. Also, $I_3\to 0$, since $\frac{\nabla K}{K}\in C(\Omega)$. Concerning $I_4$, we have
\begin{align}
    &\nonumber I_4=\int_{B_{   \epsl}} \frac12\sqrt{A(x,\sqrt{K(x)}\nabla(\psi\circ u_{\epsl}))}H_{\epsl}\diff x\diff t\\&=\frac12\int_0^T\int_\Omega \left(2Id-\frac{\nabla (\psi\circ u_{\epsl})}{\vert \nabla(\psi\circ u_{\epsl})\vert_\phi}\otimes {A_p\left(x,\frac{\nabla (\psi\circ u_{\epsl})}{\vert \nabla(\psi\circ u_{\epsl})\vert_\phi}\right)}\right):\nabla \xi\ d\vert \sqrt{K(x)}\nabla(\psi\circ u_{\epsl})\vert_\phi\ \diff t \nonumber\\&+\frac12\int_0^T\int_\Omega{ A_x\left(x,\frac{\nabla (\psi\circ u_{\epsl})}{\vert \nabla(\psi\circ u_{\epsl})\vert_\phi}\right)}\cdot \xi\ d\vert \sqrt{K(x)}\nabla(\psi\circ u_{\epsl})\vert_\phi \ \diff t\nonumber\\&
    +\frac12\int_0^T\int_\Omega\frac{\nabla K(x)}{K(x)}\cdot \xi\ d\vert \sqrt{K(x)}\nabla(\psi\circ u_{\epsl})\vert_\phi \ \diff t
    .\label{toconv}
\end{align}
We want now to apply Lemma \ref{Cic} to conclude the convergence argument. 
Let us define 
\begin{align*}
&F(t,x,p)\\&:=\left[\left(2Id-\frac{p}{\phi^{\circ}(x,p)}\otimes  {A_p\left(x,\frac{p}{\phi^{\circ}(x,p)}\right)}\right):\nabla \xi(t,x)\right.\\&\quad\left.+{ A_x\left(x,\frac{p}{\phi^{\circ}(x,p)}\right)}\cdot \xi(t,x)+\frac{\nabla K(x)}{K(x)}\cdot \xi(t,x)\right]\phi^{\circ}(x,p),
\end{align*}
so that from \eqref{toconv} we have
\begin{align*}
    I_4=\frac12\int_0^T\int_\Omega F(t,x,\sqrt{K(x)}\nabla(\psi\circ u_{\epsl}))\diff x\diff t .
\end{align*}
Observe that, recalling \eqref{convA} and \eqref{toshow} and also the proof of Lemma \ref{base}, all the assumptions of Lemmas \ref{technical} and \ref{Cic} are satisfied by choosing $\{v_k\}_k:=\{\psi\circ u_{\epsl}\}_l$, $v_0:=c_0\chi_{\widetilde{\Omega}}$, and $F$ as above (note that $\xi$ has compact support in $(0,T)\times\Omega$), so that we immediately get from \eqref{fin} that 
\begin{align*}
    I_4\to \frac12\int_0^T\int_\Omega F\left(t,x,\frac{\nu_{v_0}}{\phi^{\circ}(x,\nu_{v_0})}\right)\vert \sqrt{K(x)}Dv_0\vert_\phi\ \diff t,\quad \text{as }l\to \infty, 
\end{align*}
where $\nu_{v_0}=\frac{Dv_0}{\vert Dv_0\vert}=\nu$, with $\nu$ as the inner normal to $\partial^*\Omega(t)$. Rewriting the final limit, we have, recalling that $A_p=2\phi^{\circ}\ \phi^{\circ}_p$ and $A_x=2\phi^{\circ}\ \phi^{\circ}_x$,
\begin{align*}
    &\frac12\int_0^T\int_\Omega F(t,x,\frac{\nu_{v_0}}{\phi^{\circ}(x,\nu_{v_0})})d\vert Dv_0\vert_\phi\ \diff t\\&
    =\frac12c_0\int_0^T\int_\Omega \left[2\left(Id-\nu_\phi\otimes  n_\phi\right):\nabla \xi+2{ \phi^{\circ}_x\left(x,\nu_\phi\right)}\cdot \xi+\frac{\nabla K(x)}{K(x)}\right]\sqrt{K(x)}\phi^{\circ}(x,\nu)d \vert D\chi_{\Omega(t)}\vert\ \diff t.
\end{align*}
Coming back to \eqref{jeps}, since also $j_{\eps_l}{\rightharpoonup} j$ weakly in $L^1((0,T)\times\Omega)$, we can pass to the limit in the subsequence $\epsl$ as $l\to \infty$, to obtain in the end \eqref{2bis}. 
In conclusion, \eqref{per} comes directly from  \eqref{energyestimate1}, \eqref{eq:limit_inferior_j}, \eqref{controlper}, and the convergence assumption on $E(u_{\eps,0})$, since
\begin{align*}
     &c_0\mathcal P_\phi^K(\Omega(T))+\int_0^T\int_{\Omega(t)} \vert j\vert^2\diff x\diff t\\&\leq \liminf_{l\to\infty}\left(E_{\epsl}(u_{\epsl})+\int_0^T\int_{\Omega} \frac{\vert j_{\epsl}\vert^2}{u_{\epsl}}\diff x\diff t\right)\\&\leq \liminf_{l\to\infty}E(u_{\epsl,0})=c_0\mathcal P_\phi^K(\Omega_0).
\end{align*}
The proof of Theorem \ref{conv} is thus concluded.
\subsubsection{Method II. Anisotropic tilt excess approach
}

We propose here a second approach to prove the convergence of  \eqref{jeps} as $l\to\infty$. In particular, we prove the convergence of the right-hand side by means of an anisotropic tilt excess approach, analogously to \cite{LauxUllrich}. To this aim, we need to use the uniform convexity of $A$ which was not needed in the previous proof (our only use was the existence of weak solutions). 

Additionally, we need a further assumption on a Lipschitz property of the second argument of $\phi^{\circ}_x$, namely that there exists $C>0$ such that
\begin{align}
\abs{\phi^{\circ}_x(x,p)-\phi^{\circ}_x(x,q)}\leq C\abs{p-q},     \quad\forall p,q\in \R^d, \quad \forall x\in \Omega.
    \label{Ax}
\end{align}

We now introduce the cutoff function $\widetilde{\psi}$  as $\widetilde\psi\in C^\infty([0,\infty))$ such that
\begin{align}
\widetilde \psi(r)\equiv 0,\quad r\leq \frac14,\quad \widetilde\psi(r)=1,\quad r\geq \frac12,\quad \widetilde\psi'\geq 0.
    \label{psit}
\end{align}
Thanks to the strong convexity assumption, we can reproduce, arguing pointwise for any $x\in \Omega$, the proof of \cite[Lemma 2.4]{LauxUllrich} to obtain 
\begin{lemma}
    There exist constants $c_\phi^{\circ},C_\phi^{\circ}>0$, depending only on $\phi^{\circ}$, such that 
    \begin{align}
       \phi^{\circ}(x,p)-\abs{p'}\widetilde
       \psi(\abs{p'})\phi^{\circ}_p(x,p')\cdot p\geq c_\phi^{\circ}\abs{p-p'}^2,    \label{p1}\end{align}
       for all $p,p'\in \R^d$ such that $\abs{p}=1$ and $\abs{p'}\leq 1$, and for any $x\in \Omega$. Also, it holds 
       \begin{align}
           \phi^{\circ}(x,p)-\abs{p'}\widetilde\psi(\abs{p'})\phi^{\circ}_p(x,p')\cdot p\leq C_\phi^{\circ} (\abs{p-p'}^2+(1-\abs{p'})),\label{p2}
           \end{align}
              for all $p,p'\in \R^d$ such that $\abs{p}=1$ and $\abs{p'}\leq 1$, and for any $x\in \Omega$.
\end{lemma}
We can now define the notion of relative entropy.  Given $ u\in BV (\Omega;\{0,1\})$, let $\nu=\frac{D u}{\abs{Du}}$
be the measure theoretic inner unit normal. Recalling the definition of \eqref{psit}, the relative entropy of $u$ with respect to a vector field $\zeta \in C(\Omega)^d$
is
\begin{align}
\mathcal E[u|\zeta]:=c_0\int_\Omega (\phi^{\circ}(x,\nu)-\abs{\zeta}\widetilde{\psi}(\abs{\zeta})\phi^{\circ}_p(x,\zeta)\cdot\nu)\sqrt{K(x)}d\abs{Du}.
    \label{relentr}
\end{align}
Analogously, given $\eps>0$ and $u_\eps\in H^1(\Omega)$, and defining the approximated inner unit normal as
\begin{align}
    \nu_\eps:=\begin{cases}
        \frac{\nabla u_\eps}{\abs{\nabla u_\eps}},\quad \text{ if }\nabla u_\eps\not=0,\\
        e_1,\quad \text{ if } \nabla u_\eps=0,
    \end{cases}
    \label{nueps}
\end{align}
the $\eps$-relative entropy of $u_\eps$ with respect to a vector field $\zeta \in C(\Omega)^d$
is
\begin{align}
\nonumber\mathcal E_\eps[u_\eps|\zeta]&:=2\int_\Omega (\phi^{\circ}(x,\nabla u_\eps)-\abs{\zeta}\widetilde\psi(\abs{\zeta})\phi^{\circ}_p(x,\zeta)\cdot\nabla u_\eps)\sqrt{K(x)W(u_\eps)}d x\\&
\ =2\int_\Omega (\phi^{\circ}(x,\nu_\eps)-\abs{\zeta}\widetilde\psi(\abs{\zeta})\phi^{\circ}_p(x,\zeta)\cdot\nu_\eps)\sqrt{K(x)W(u_\eps)}\abs{\nabla u_\eps}d x.
    \label{relentreps}
\end{align}
Then, recalling Lemma \ref{base}, we can follow the same arguments as in \cite[Lemma 4.7]{LauxUllrich}, exploiting \eqref{p1}, to infer that, under the assumptions of Theorem \ref{conv}, it holds, for the same subsequence $\{\eps_l\}_l$, 
\begin{align}
    \lim_{l\to \infty} \int_0^T\mathcal E_{\eps_l}[u_{\eps_l}(t)|\zeta(t)]d t=\int_0^T\mathcal E[c_0\chi_{\widetilde{\Omega}}(t)|\zeta(t)]\diff t,\label{convergence1}
\end{align}
for every $\zeta\in C([0,T]\times\Omega)^d$ such that $\abs{\zeta}\leq1$.
Then, by means of \eqref{p2}, we can also show that that the tilt excess can be made arbitrarily small by approximating the
normal $\nu$ with suitable continuous vector fields $\zeta$. Namely, by arguing as in the proof of \cite[Lemma 4.8]{LauxUllrich}, recalling that $K\in C^1(\Omega)$ and $0<K_*\leq K(x)\leq K^*$ for any $x\in \Omega$, we infer that for any $\delta>0$ there exists a smooth vector $\zeta\in C^1([0,T]\times\Omega)^d$ such that $\abs{\zeta}\leq1$ and 
\begin{align}
\int_0^T\mathcal E[c_0\chi_{\widetilde{\Omega}}(t)|\zeta(t)]d t<\delta.
\label{smallunif}
\end{align}
Now we can complete the proof. In particular, it is immediate to see that, following the same lines of the proof in \cite{LauxUllrich}, we can prove, exploiting the above results and Lemma \ref{base}, that
\begin{align*}
    \int_0^T\int_\Omega\mathbf T_{\eps_l}:\nabla \xi \diff x \diff t\to c_0\int_0^T\int_\Omega \left(Id-\nu_\phi\otimes  n_\phi\right):\nabla \xi\sqrt{K(x)}\phi^{\circ}(x,\nu)d \vert D\chi_{\Omega(t)}\vert\ \diff t,
\end{align*}
and we therefore omit the proof, referring to \cite[Lemma 4.8]{LauxUllrich} for the details.
We thus concentrate on the two new terms appearing in the right-hand side of \eqref{jeps}. First, given $\zeta\in C([0,T]\times\Omega)^d$, we have, since $\phi^{\circ}_x$ is still 1-homogeneous in its second argument, 
\begin{align*}
    &\int_0^T\int_\Omega \eps_lA_x(x,\nabla u_{\eps_l})\cdot \xi\  \diff x\ \diff t\\&=2\int_0^T\int_\Omega \phi^{\circ}_x\left(x,\frac{\nabla u_{\eps_l}}{\phi^{\circ}(x,\nabla u_{\eps_l})}\right)\cdot \xi\ \eps_l A\left(x,{\nabla u_{\eps_l}}\right) \diff x\ \diff t\\&
    =2\int_0^T\int_\Omega \phi^{\circ}_x\left(x,\frac{\nabla u_{\eps_l}}{\phi^{\circ}(x,\nabla u_{\eps_l})}\right)\cdot \xi\ \left(\eps_l A\left(x,{\nabla u_{\eps_l}}\right)-\frac12\phi^{\circ}(x,\sqrt{K(x)}\nabla(\psi\circ u_{\eps_l}))\right) \diff x\ \diff t\\&
    + \int_0^T\int_\Omega \phi^{\circ}_x\left(x,\frac{\nabla u_{\eps_l}}{\phi^{\circ}(x,\nabla u_{\eps_l})}\right)\cdot \xi\ \sqrt{K(x)}\phi^{\circ}(x,\nabla(\psi\circ u_{\eps_l})) \diff x\ \diff t,
\end{align*}
and observe that the first term converges to zero as $l\to\infty$, since $\phi^{\circ}_x\left(x,\frac{\nabla u_{\eps_l}}{\phi^{\circ}(x,\nabla u_{\eps_l})}\right)\leq C$ (indeed, $\phi^{\circ}_x/\phi^{\circ}$ is 0-homogeneous in the second argument) and exploiting the convergence \eqref{essentialsquareroot}. Concerning the second term, we can compare it with its expected limit, to obtain, for any $\zeta\in C([0,T]\times\Omega)^d$, $\abs{\zeta}\leq1$,
\begin{align*}
    &\int_0^T\int_\Omega \phi^{\circ}_x\left(x,\frac{\nabla u_{\eps_l}}{\phi^{\circ}(x,\nabla u_{\eps_l})}\right)\cdot \xi\ \sqrt{K(x)}\phi^{\circ}(x,\nabla(\psi\circ u_{\eps_l})) \diff x\ \diff t\\&-c_0\int_0^T\int_\Omega \phi^{\circ}_x(x,\nu_\phi)\cdot \xi{\sqrt K(x)}\phi^{\circ}(x,\nu)d\abs{D\chi_{\widetilde{\Omega}}(t)}\diff t\\&
    =\underbrace{2\int_0^T\int_\Omega \phi^{\circ}_x(x,\nu_{\eps_l})\cdot \xi\ \sqrt{K(x)W(u_{\eps_l})}\abs{\nabla u_{\eps_l}} \diff x\ \diff t-2\int_0^T\int_\Omega \phi^{\circ}_x(x,\zeta)\cdot \xi\ \sqrt{K(x)W(u_{\eps_l})}\abs{\nabla u_{\eps_l}} \diff x\ \diff t}_{J_1}\\&
    \underbrace{+2\int_0^T\int_\Omega \phi^{\circ}_x(x,\zeta)\cdot \xi\ \sqrt{K(x)W(u_{\eps_l})}\abs{\nabla u_{\eps_l}} \diff x\ \diff t
    -c_0\int_0^T\int_\Omega \phi^{\circ}_x(x,\zeta)\cdot \xi{\sqrt K(x)}d\abs{D\chi_{\widetilde{\Omega}}(t)}\diff t}_{J_2}\\&
    \underbrace{+c_0\int_0^T\int_\Omega \phi^{\circ}_x(x,\zeta)\cdot \xi{\sqrt K(x)}d\abs{D\chi_{\widetilde{\Omega}}(t)}\diff t-c_0\int_0^T\int_\Omega \phi^{\circ}_x(x,\nu)\cdot \xi{\sqrt K(x)}d\abs{D\chi_{\widetilde{\Omega}}(t)}\diff t}_{J_3}.
\end{align*}
We now estimate each $J_i$ for $i=1,2,3$. Let us start from $J_2$, since $\zeta\in C([0,T]\times\Omega)^d$ and $K\in C^1(\Omega)$, the fact that $J_2\to 0$ as $l\to \infty$ directly comes from \eqref{convA}.
We now pass to $J_1$ and $J_3$. It holds, recalling \eqref{bound}, \eqref{Ax} and \eqref{p1}, the regularity of $\xi$, and Cauchy-Schwarz inequality,
\begin{align*}
    \abs{J_1}&\leq  2\int_0^T\int_\Omega \abs{\phi^{\circ}_x(x,\nu_{\eps_l})-\phi^{\circ}_x(x,\zeta)}\abs{\xi}\ \sqrt{K(x)W(u_{\eps_l})}\abs{\nabla u_{\eps_l}} \diff x\diff t
    \\&
    \leq C\int_0^T\int_\Omega \abs{\phi^{\circ}_x(x,\nu_{\eps_l})-\phi^{\circ}_x(x,\zeta)}\abs{\xi}\ \sqrt{K(x)W(u_{\eps_l})}\abs{\nabla u_{\eps_l}} \diff x\diff t \\&
    \leq C\int_0^T\int_\Omega \abs{\nu_{\eps_l}-\zeta}\ \sqrt{K(x)W(u_{\eps_l})}\abs{\nabla u_{\eps_l}} \diff x\diff t\\&
    \leq C\left(\int_0^T\int_\Omega \sqrt{K(x)W(u_{\eps_l})}\abs{\nabla u_{\eps_l}} \diff x\diff t\right)^\frac12\left(\int_0^T\int_\Omega \abs{\nu_{\eps_l}-\zeta}^2\ \sqrt{K(x)W(u_{\eps_l})}\abs{\nabla u_{\eps_l}} \diff x\diff t\right)^\frac12
    \\& \leq C(T)\left(\int_0^T\mathcal E_{\eps_l}[u_{\eps_l}(t)|\zeta(t)]\diff t\right)^\frac12.
\end{align*}
Similarly, about $J_3$, by \eqref{energyestimate1}, \eqref{controlper}, \eqref{Ax}, \eqref{p1}, and Cauchy-Schwarz inequality,
\begin{align*}
    &\abs{J_3}\leq C\left(\int_0^T\int_\Omega \sqrt K(x) d\abs{D\chi_{\widetilde{\Omega}}(t)}\diff t\right)^\frac12\left(\int_0^T\int_\Omega \abs{\nu-\zeta}^2\sqrt K(x) d\abs{D\chi_{\widetilde{\Omega}}(t)}\diff t\right)^\frac12\\&
    \leq C\left(\int_0^T\mathcal P_\phi^K(\Omega(t))\diff t\right)^\frac12\left(\int_0^T\mathcal E[c_0\chi_{\widetilde{\Omega}}(t)|\zeta]\diff t\right)^\frac12\leq C(T)\left(\int_0^T\mathcal E[c_0\chi_{\widetilde{\Omega}}(t)|\zeta]\diff t\right)^\frac12,
\end{align*}
so that, by taking the limit as $l\to       \infty$ and recalling \eqref{convergence1}, we get
    \begin{align}
        \limsup_{l\to \infty}(\abs{J_1}(\zeta)+\abs{J_3}(\zeta))\leq C(T)\left(\int_0^T\mathcal E[c_0\chi_{\widetilde{\Omega}}(t)|\zeta(t)]\diff t\right)^\frac12.\label{reg1}
    \end{align}
Then, by the arbitrariness of $\zeta$, we infer from \eqref{smallunif} that the right-hand side in \eqref{reg1} can be made arbitrarily small.

We can thus conclude that 
\begin{align*}
&\int_0^T\int_\Omega \phi^{\circ}_x\left(x,\frac{\nabla u_{\eps_l}}{\phi^{\circ}(x,\nabla u_{\eps_l})}\right)\cdot \xi\ \sqrt{K(x)}\phi^{\circ}(x,\nabla(\psi\circ u_{\eps_l})) \diff x\ \diff t\\&\to c_0\int_0^T\int_\Omega \phi^{\circ}_x(x,\nu_\phi)\cdot \xi{\sqrt K(x)}\phi^{\circ}(x,\nu)d\abs{D\chi_{\widetilde{\Omega}}(t)}\diff t,
\end{align*}
as $l\to \infty$. 

We are left to consider the last term in the right-hand side of \eqref{jeps}. Namely, we have
\begin{align*}
    &\int_0^T\int_\Omega \frac1{\eps_l}W(u_{\eps_l})\nabla K(x)\cdot \xi \diff x \diff t\\&= \int_0^T\int_\Omega \frac1{\eps_l}K(x)W(u_{\eps_l})\frac{\nabla K(x)}{{K(x)}}\cdot \xi \diff x \diff t\\&
    =\int_0^T\int_\Omega \left(\frac1{\eps_l}K(x)W(u_{\eps_l})-\frac12\sqrt{K(x)}\phi^{\circ}(x,\nabla(\psi\circ u_{\eps_l}))\right)\frac{\nabla K(x)}{{K(x)}}\cdot \xi \diff x \diff t\\&+\int_0^T\int_\Omega \sqrt{K(x)}\phi^{\circ}(x,\nabla(\psi\circ u_{\eps_l}))\frac{\nabla K(x)}{2{K(x)}}\cdot \xi \diff x \diff t,
\end{align*}
and again the first term in the right-hand side converges to zero as $l\to\infty$ due to \eqref{essentialsquareroot1} and the regularity of $K$. In conclusion, concerning the last summand, we first note that, as $l\to\infty$,
\begin{align}
\mathcal L^1_{|(0,T)}\otimes \phi^{\circ}(\cdot,\sqrt K\nabla(\psi\circ u_{\eps_l}))\overset{*}{\rightharpoonup}\mathcal L^1_{|(0,T)}\otimes 
 c_0 \sqrt K\phi^{\circ}(\cdot,\nu)d\abs{D \chi_{\widetilde{\Omega}}(t)},\quad\text{ weakly* in }\mathcal M((0,T)\times\Omega).
\label{lastconv}
\end{align}
\begin{proof}[Proof of \eqref{lastconv}]
By a standard Fubini-like argument from \eqref{controlper0}, which also holds for any open set in place of $\Omega$ (restricting the perimeter on that set), it is easy to deduce the validity of \eqref{lastconv}.
On the other hand this can also be obtained by means a tilt-excess argument, which we detail here. For any $\rho \in C([0,T]\times\Omega)^d$ and for any $\zeta\in C([0,T]\times\Omega)^d$ such that $\abs{\zeta}\leq 1$, we have
\begin{align*}
    &\int_0^T\int_\Omega \rho(t,x)\cdot \phi^{\circ}(x,\sqrt {K(x)}\nabla(\psi\circ u_{\eps_l}))\diff x\diff t-c_0\int_0^T\int_\Omega \rho(t,x)\sqrt {K(x)} \phi^{\circ}(x,\nu)d\abs{D \chi_{\widetilde{\Omega}}(t)}\\&
    =\\&\underbrace{2\int_0^T\int_\Omega \rho(t,x)\phi^{\circ}(x,\nu_{\eps_l})\cdot \sqrt {K(x)W(u_{\eps_l})}\abs{\nabla u_{\eps_l}}\diff x\diff t-2\int_0^T\int_\Omega \rho(t,x)\phi^{\circ}(x,\zeta)\sqrt {K(x)W(u_{\eps_l})} \diff x \diff t}_{J_4}\\&
    +\underbrace{2\int_0^T\int_\Omega \rho(t,x)\phi^{\circ}(x,\zeta)\sqrt {K(x)W(u_{\eps_l})} \diff x \diff t-c_0\int_0^T\int_\Omega \rho(t,x)\sqrt {K(x)} \phi^{\circ}(x,\zeta)d\abs{D \chi_{\widetilde{\Omega}}(t)}}_{J_5}\\&
    +\underbrace{c_0\int_0^T\int_\Omega \rho(t,x)\sqrt {K(x)} \phi^{\circ}(x,\zeta)d\abs{D \chi_{\widetilde{\Omega}}(t)}-c_0\int_0^T\int_\Omega \rho(t,x)\sqrt {K(x)} \phi^{\circ}(x,\nu)d\abs{D \chi_{\widetilde{\Omega}}(t)}}_{J_6}.
\end{align*}
Concerning $J_5\to 0$, this directly comes from \eqref{convA}.
Now, since $\phi^{\circ}$ is Lipschitz continuous in its second argument due to assumption \eqref{ass:A}, using the same estimates as for $J_1$, we get
\begin{align*}
    \abs{J_4}&\leq  2\int_0^T\int_\Omega \abs{\phi^{\circ}(x,\nu_{\eps_l})-\phi^{\circ}(x,\zeta)}\abs{\rho}\ \sqrt{K(x)W(u_{\eps_l})}\abs{\nabla u_{\eps_l}} \diff x\diff t
    \\&
    \leq C\int_0^T\int_\Omega \abs{\nu_{\eps_l}-\zeta}\ \sqrt{K(x)W(u_{\eps_l})}\abs{\nabla u_{\eps_l}} \diff x\diff t
    \leq C(T)\left(\int_0^T\mathcal E_{\eps_l}[u_{\eps_l}(t)|\zeta(t)]\diff t\right)^\frac12.
\end{align*}
In conclusion, about $J_6$, recalling again the Lipschitz property of $\phi^{\circ}$ in its second argument, and arguing as for $J_3$,
\begin{align*}
    &\abs{J_6}\leq C\left(\int_0^T\int_\Omega \sqrt K(x) d\abs{D\chi_{\widetilde{\Omega}}(t)}\diff t\right)^\frac12\left(\int_0^T\int_\Omega \abs{\nu-\zeta}^2\sqrt K(x) d\abs{D\chi_{\widetilde{\Omega}}(t)}\diff t\right)^\frac12\\&
    \leq C\left(\int_0^T\mathcal P_\phi^K(\Omega(t))\diff t\right)^\frac12\left(\int_0^T\mathcal E[c_0\chi_{\widetilde{\Omega}}(t)|\zeta]\diff t\right)^\frac12\leq C(T)\left(\int_0^T\mathcal E[c_0\chi_{\widetilde{\Omega}}(t)|\zeta]\diff t\right)^\frac12,
\end{align*}
As a consequence, we can write, recalling \eqref{convergence1}, 
   \begin{align*}
        \limsup_{l\to \infty}(\abs{J_4}(\zeta)+\abs{J_6}(\zeta))\leq C(T)\left(\int_0^T\mathcal E[c_0\chi_{\widetilde{\Omega}}(t)|\zeta(t)]\diff t\right)^\frac12.
    \end{align*}
Therefore by the arbitrariness of $\zeta$,  from \eqref{smallunif} we deduce that the quantity above can be made arbitrarily small. This proves \eqref{lastconv}.
\end{proof}

Therefore, thanks to the regularity of $K$ and $\xi$, we immediately infer 
\begin{align*}
    &\int_0^T\int_\Omega \frac{\nabla K(x)}{2K(x)}\cdot \xi\ \sqrt{K(x)}\phi^{\circ}(x,\nabla(\psi\circ u_{\eps_l})) \diff x\ \diff t\\&\to c_0\int_0^T\int_\Omega \frac{\nabla K(x)}{2K(x)}\cdot \xi{\sqrt K(x)}\phi^{\circ}(x,\nu)d\abs{D\chi_{\widetilde{\Omega}}(t)}\diff t.
\end{align*}
This allows to show that 
$$
\int_0^T\int_\Omega \frac1{\eps_l}W(u_{\eps_l})\nabla K(x)\cdot \xi \diff x \diff t\to c_0\int_0^T\int_\Omega {\nabla \sqrt K(x)}\cdot \xi\phi^{\circ}(x,\nu)d\abs{D\chi_{\widetilde{\Omega}}(t)}\diff t,
$$
as $l\to \infty$.

As a consequence, since also $j_{\eps_l}{\rightharpoonup} j$ weakly in $L^1((0,T)\times\Omega)$, we can pass to the limit as $l\to \infty$ in  \eqref{jeps}, to obtain again in the end \eqref{2bis}. This concludes the argument concerning the convergence in \eqref{jeps} by means of anisotropic tilt excess.
\subsection{Weak solutions to anisotropic weighted Hele-Shaw flow are strong solutions}
Here we show that, if $j$ and $\widetilde{\Omega}:=
\bigcup_{t\in[0,T]}\Omega(t)\times\{t\}$ are sufficiently smooth and additionally they are a weak solution to the anisotropic Hele–Shaw flow according to Definition \ref{heleshaw}, then $\widetilde{\Omega}$ and $j$ solve the anisotropic weighted Hele–Shaw equations in the classical sense, according to Definition \ref{classicalHeleshaw}.
In particular, we have the following result
\begin{lemma}
    Let $(\widetilde{\Omega}, j)$ be a weak solution to the anisotropic weighted Hele-Shaw flow in the sense of Definition \ref{heleshaw}. If $j$ is sufficiently smooth and $\Omega(t)$ evolves smoothly and is simply connected for all t, then $(\widetilde{\Omega}, j)$ is also a classical solution to the anisotropic weighted Hele–Shaw flow \eqref{a1}-\eqref{a2}.
\end{lemma}
\begin{proof}
 The proof of the validity of \eqref{a1} for $(j,\widetilde{\Omega})$ can be obtained as in Step 1 of the proof of \cite[Lemma 4.10]{laux_kroemer}. Concerning \eqref{a2}, we can follow Step 2 of the same lemma, adapting some parts. Let us consider $\xi\in C^1_c(\Omega(t);\mathbb R^d)$ such that $\Div \xi=0$. Thanks to \eqref{2}, we have  
\begin{align}       
&\int_0^T\int_{\Omega(t)}\xi\cdot j(\cdot,t)\diff x \ \diff t\\&=c_0\int_0^T\int_{\partial\Omega(t)}tr\left[(Id-n_\phi\otimes \nu_\phi)\nabla \xi+\left(\phi^{\circ}_x(x,\nu_\phi)+\frac{\nabla K(x)}{2K(x)}\right)\otimes \xi\right]\sqrt{K}(x)\phi^{\circ}(x,\nu) d\mathcal H^{d-1} \ \diff t=0,\label{2b}
        \end{align}
since $\xi$ has compact support inside $\Omega(t)$, so that 
$$
\int_0^T\int_{\Omega(t)}j\cdot \xi \diff x \diff t=0.
$$
This entails by a density argument that, for any $t\in[0,T]$ (recall that $\Omega(t)$ is a smooth function of $t$), 
$$
\int_{\Omega(t)}j\cdot \xi \diff x=0, \quad\forall \xi\in L^2(\Omega(t)),\quad \Div \xi =0, 
$$
implying that $j\perp_{L^2(\Omega(t))}\{\xi\in L^2(\Omega(t)):\ \Div\xi=0\}$, so that, since $\Omega(t)$ is of class $C^{2,\alpha}$, $\alpha\in(0,1]$, and simply connected, by the Helmholtz decomposition (see, e.g., \cite[Chap. IV]{boyer}) there exists $p(t)\in H^1(\Omega(t))$ such that 
$$
j=-\nabla p(t),\quad\text{ in }\Omega(t),\quad \forall t\in[0,T].
$$
Plugging this result in \eqref{2}, we infer, recalling Theorem \ref{basic1} and the definition of $\Div_\phi$, 
\begin{align*}
   &- \int_0^T\int_{\Omega(t)}\nabla p(t)\cdot \xi \diff x\diff t \\&=c_0\int_0^T\int_{\partial\Omega(t)}\Div_\phi(\xi\sqrt K)\phi^{\circ}(x,\nu)d\mathcal H^{d-1} \diff t\\&\quad +c_0\int_0^T\int_{\partial\Omega(t)}({\nabla \sqrt{K(x)}}\cdot n_\phi) (\xi\cdot \nu_\phi)\phi^{\circ}(x,\nu)d\mathcal H^{d-1}\diff t\\&=c_0\int_0^T\int_{\partial\Omega(t)}\left(-H_\phi\sqrt {K(x)}+{\nabla \sqrt{K(x)}}\cdot n_\phi\right) (\xi\cdot \nu_\phi)\phi^{\circ}(x,\nu) \ d\mathcal H^{d-1}\diff t\\&
=c_0\int_0^T\int_{\partial\Omega(t)}\left(-H_\phi\sqrt{K(x)}+{\nabla \sqrt{K(x)}}\cdot n_\phi\right)\nu\cdot \xi \ d\mathcal H^{d-1}\diff t,
\end{align*}
recalling that $\nu_\phi= \f{\nu}{\phi^{\circ}(x,\nu)}$. Observing also that, by the divergence theorem,
\begin{align*}
    - \int_0^T\int_{\Omega(t)}\nabla p(t)\cdot \xi \diff x\diff t=\int_0^T\int_{\partial\Omega(t)}p(t) \xi\cdot\nu d\mathcal H^{d-1}\diff t,
\end{align*}
we obtain by the fundamental lemma of calculus of variations that $$p(t)=-c_0H_\phi(t)\sqrt{K}+c_0{\nabla \sqrt{K}}\cdot n_\phi$$ on $\partial\Omega(t)$ for any $t\in[0,T]$. This concludes the proof that $(j,\widetilde{\Omega})$ also satisfies \eqref{a2}.
\end{proof}

\ \\ \ \\ \
\textbf{Acknowledgments.} We thank the anonymous referees for their valuable comments and remarks, which significantly improved the clarity of our work. Part of this contribution was completed while CE was visiting AP at the Faculty  of Mathematics of the University of Vienna, whose hospitality is kindly acknowledged.
AP is a member of Gruppo Nazionale per l’Analisi Matematica, la Probabilità e le loro Applicazioni (GNAMPA) of
Istituto Nazionale per l’Alta Matematica (INdAM). AP also gratefully aknowledges support from the Alexander von Humboldt Foundation. This research was funded in part by the Austrian Science Fund (FWF) \href{https://doi.org/10.55776/ESP552}{10.55776/ESP552}. CE was supported by the European Union via the ERC
AdG 101054420 EYAWKAJKOS project.
For open access purposes, the authors have applied a CC BY public copyright license to
any author accepted manuscript version arising from this submission. 
\\

\begin{appendices}
\renewcommand{\thesection}{\Alph{section}.}
 \renewcommand{\theequation}{\thesection\arabic{equation}}
\section{Proof of Lemma \ref{Cic}}
Here we propose the proof of Lemma \ref{Cic}, which is a generalization of \cite[Lemma 3.7]{CicaleseNagasePisante}.
\label{Appendix}
Let $\theta_k: (0,T)\times\Omega\to \{(t,x,p): \ x\in \Omega,\ t\in(0,T),\ p\in \partial B_{\phi^\circ(x)}\}$ be defined by 
\begin{align*}
    \theta_k(t,x):=\left(t,x,\frac{Dv_k(t,x)}{\phi^{\circ}(x,Dv_k(t,x))}\right),
\end{align*}
which is well defined since $Dv_k$ is an $L^1(\Omega)$ function. We consider the family of pushforward measures $\mu_k:=(\theta_k)_\sharp(\vert\sqrt{K(x)} Dv_k\vert_\phi\mathcal L^{d+1}_{|(0,T)\times\Omega})=(\theta_k)_\sharp\left(\phi^{\circ}\left(x,\sqrt{K(x)}Dv_k(t,x)\right)\mathcal L^{d+1}_{|(0,T)\times\Omega}\right)$. Note that the measures $\mu_k$ are defined over the set $S:=\bigcup_{(t,x)\in(0,T)\times\Omega}\{(t,x)\}\times \partial B_{\phi^\circ(x)}$. By the definition of $\mu_k$, it holds 
\begin{align}
&\nonumber\int_{S} \xi(t,x,p) d\mu_k\\&=\int_{(0,T)\times\Omega}\xi\left(t,x,\frac{Dv_k(x)}{\phi^{\circ}(x,Dv_k(t,x))}\right)\phi^{\circ}(x,\sqrt{K(x)}Dv_k(t,x))\diff x\diff t
    \label{integral}
\end{align}
for any $\xi\in C_c(S)$. we also consider
$$
\theta_0(t,x):=\left(t,x,\frac{\nu_{v_0}(t,x)}{\phi^{\circ}(x,\nu_{v_0}(t,x))}\right),
$$
with $\nu_{v_0}=\frac{Dv_0}{\vert Dv_0\vert}$, i.e., $\nu_{v_0}$ is the Radon-Nykodym derivative of $Dv_0$ with respect to its total variation. We then introduce
$$\mu_0:=(\theta_0)_\sharp (\mathcal L^1_{\vert (0,T)}\otimes \vert \sqrt{K(x)} D v_0(t)\vert_\phi)=(\theta_0)_\sharp (\mathcal L^1_{\vert(0,T)}\otimes \phi^{\circ}(x,\nu_{v_0}(t,x))\vert \sqrt{K(x)} D v_0(t)\vert),$$ so that we have 
\begin{align}
\int_{S} \xi(t,x,p) d\mu_0=\int_{(0,T)\times\Omega}\xi\left(t,x,\frac{\nu_{v_0}(t,x)}{\phi^{\circ}(x,\nu_{v_0}(t,x))}\right)\phi^{\circ}(x,\nu_{v_0}(t,x))d\left\vert \sqrt{K(x)} Dv_0(t)\right\vert\diff t
    \label{integral}
\end{align}
for any $\xi\in C_c(S)$. By the regularity and assumption \eqref{van} of $F$, we see that $F$ is an admissible test function, so that the identities above also hold with $\xi=F$. Therefore, recalling assumption \eqref{fond} and the fact that $F$ is positively 1-homogeneous with respect to the third argument $p$, the claim \eqref{fond} follows if we show that $\mu_k\overset{*}{\rightharpoonup} \mu_0$ in  measure as $k\to\infty$. Now, note that, by assumption \eqref{fond}, it holds
\begin{align*}
    \mu_k(S)=\int_{(0,T)\times\Omega}\phi^{\circ}(x,\sqrt{K(x)}Dv_k(t,x))\diff x\diff t\leq C,\quad \forall k\in\mathbb N.
\end{align*}
Therefore, since $\mu_k$ is a nonnegative Radon measure, there exists a Radon measure $\mu$ such that $\mu_k\overset{*}{\rightharpoonup} \mu$ in measure as $k\to \infty$. We only need to show that $\mu=\mu_0$. By the disintegration of measures theorem applied to $\mu_k$ (see, e.g., \cite[Theorem 2.4]{CicaleseNagasePisante}), we can set $X=S$, $Y=(0,T)\times\Omega$ and consider the Borel map $\pi: X\to Y$ such that $\pi(t,x,p)=(t,x)$, setting $\omega=\pi_\sharp \mu$, so that the exists the nonnegative Radon (probability) measure $\lambda_{t,x}$ defined on the space $\{(t,x)\}\times \partial B_{\phi^\circ(x)}$ for $\omega$-a.a. $(t,x)\in (0,T)\times\Omega$,  such that for every Borel function $f: X\to[0,+\infty]$, it holds 
\begin{align*}
    \int_S f(t,x,p)d\mu=\int_{(0,T)\times\Omega}\left(\int_{\{(t,x)\}\times \partial B_{\phi^\circ(x)}}f(t,x,p)d\lambda_{t,x}(t,x,p)\right)d\omega(t,x).
\end{align*}
Note that, since the variable of integration in the second integral is $p$, with a slight abuse of notation we simply write 
\begin{align*}
    \int_S f(t,x,p)d\mu=\int_{(0,T)\times\Omega}\left(\int_{\partial B_{\phi^\circ(x)}}f(t,x,p)d\lambda_{t,x}(p)\right)d\omega(t,x).
\end{align*}
We now need to identify $\omega=\mathcal L^1_{|(0,T)}\otimes \vert \sqrt{K(x)}Dv_0\vert_\phi(t)$, and $\lambda_{t,x}=\delta\left(\frac{\nu_{v_0}(t,x)}{\phi^{\circ}(x,\nu_{v_0}(t,x))}\right)$.

Assume $h\in C_c((0,T)\times\Omega;\mathbb R^d)$, so that we can use the factorized function $f(t,x,p)=h(t,x)\cdot p$ in the equation above and obtain 
\begin{align}
   \label{deomposition} \int_S h(t,x)\cdot p d\mu=\int_{(0,T)\times\Omega} h(t,x)\cdot \left(\int_{\partial B_{\phi^\circ(x)}}p d\lambda_{t,x}(p)\right)d\omega(t,x).
\end{align}
Now, by the weak* convergence of $\mu_k$ and the weak* convergence of $\mathcal L^1_{|(0,T)}\otimes \sqrt{K(\cdot)}Dv_k(t)$ to $\mathcal L^1_{|(0,T)}\otimes \sqrt{K(\cdot})Dv_0(t)$ in $\mathcal M((0,T)\times\Omega)$ (given by an application of Lemma \ref{technical}, recalling $K\in C(\Omega)$ and $K\geq K_*>0$), we can obtain
\begin{align*}
    &\int_S h(t,x)\cdot p d\mu\\&
    =\lim_{k\to \infty} \int_S h(t,x)\cdot p d\mu_k\\&
    =\lim_{k\to\infty} \int_0^T\int_\Omega h(t,x)\cdot \frac{Dv_k(t,x)}{\phi^{\circ}(x,Dv_k(t,x))}d\vert \sqrt{K(x)}Dv_k\vert_{\phi}\\&=\lim_{k\to \infty}\int_0^T\int_\Omega h(t,x)\cdot  \sqrt{K(x)}Dv_k(t,x)\diff x\diff t\\&
    =\lim_{k\to \infty}\int_0^T\int_\Omega h(t,x) d(\sqrt{K(x)}Dv_k)\diff t=\int_0^T\int_\Omega h(t,x) d (\sqrt{K(x)}Dv_0)\diff t.
\end{align*}
By this result and the decomposition \eqref{deomposition}, we immediately deduce, since $h$ is arbitrary, that 
$$
\left(\int_{\partial B_{\phi^\circ(x)}}p d\lambda_{t,x}(p)\right)d\omega(t,x)=d\ \mathcal L^1_{|(0,T)}\otimes \sqrt{K(x)} Dv_0(t)=\nu_{v_0}d\ 
\mathcal L^1_{|(0,T)}\otimes \vert \sqrt{K(x)} Dv_0\vert (t),
$$
entailing that $\mathcal L^1_{|(0,T)}\otimes \vert \sqrt{K(x)} Dv_0\vert_\phi(t) <<\omega$. Indeed, since $\vert \nu_{v_0}\vert\equiv 1$, we multiply both the sides of the identity above by $\phi^{\circ}(x,\nu_{v_0})\nu_{v_0}$ to obtain 
$$
\phi^{\circ}(x,\nu_{v_0}(t,x))\nu_{v_0}(t,x)\cdot \left(\int_{\partial B_{\phi^\circ(x)}}p d\lambda_{t,x}(p)\right)d\omega(t,x)=d\ \mathcal L^1_{|(0,T)}\otimes \vert \sqrt{K(x)} Dv_0\vert_\phi (t),
$$
giving the result. Thus there exists an $\omega$-measurable function $\gamma:(0,T)\times\Omega\to \mathbb R^+$ such that $ \mathcal L^1_{|(0,T)}\otimes \vert \sqrt{K(x)}Dv_0\vert_\phi(t)=\gamma \omega$. Therefore, we have
$$
\phi^{\circ}(x,\nu_{v_0})\left(\int_{\partial B_{\phi^\circ(x)}}p d\lambda_{t,x}(p)\right)d\omega(t,x)=\nu_{v_0}d\mathcal L^1_{|(0,T)}\otimes \vert \sqrt{K(x)} Dv_0\vert_\phi(t)=\nu_{v_0}\gamma d\omega(t,x),
$$
so that, for $\omega$-a.e. $ (t,x)\in (0,T)\times\Omega$, 
\begin{align}
\int_{\partial B_{\phi^\circ(x)}}p d\lambda_{t,x}(p)=\frac{\nu_{v_0}}{\phi^{\circ}(x,\nu_{v_0})}\gamma(t,x),\label{gam2}
\end{align}
which means, by applying $\phi^{\circ}$ (in the second argument) to both the sides of the equality and recalling that $\phi^{\circ}(x,p)$ is positively 1-homogeneous with respect to $p$,
\begin{align}
\phi^{\circ}\left(x,\int_{\partial B_{\phi^\circ(x)}}p d\lambda_{t,x}(p)\right)=\gamma(t,x).\label{gam1}
\end{align}
In conclusion, as in \cite[(3.37)]{CicaleseNagasePisante}, thanks to the convergence \eqref{fond} and the weak* convergence of $\mu_k$, we infer, for any $\xi\in C_c((0,T)\times\Omega)$,
\begin{align*}
   &\int_{(0,T)\times\Omega} \xi(t,x)\left(\int_{\partial B_{\phi^\circ(x)}} \phi^{\circ}(x,p)d\lambda_{t,x}(p)\right)d\omega(t,x)\\&=\int_{(0,T)\times\Omega}\xi(t,x) \left(\int_{\partial B_{\phi^\circ(x)}} d\lambda_{t,x}(p)\right)d\omega(t,x)\\&=\int_{S}\xi(t,x)d\mu=\lim_{k\to\infty}\int_S \xi(t,x)d\mu_k\\&=\lim_{k\to\infty}\int_{(0,T)\times\Omega}\xi(t,x)\phi^{\circ}(x,\sqrt{K(x)} Dv_k(t,x))\diff x\diff t\\&=\int_{(0,T)\times\Omega}\xi(t,x)\phi^{\circ}(x,\nu_{v_0})d \vert \sqrt{K(x)}Dv_0\vert \diff t\\&=\int_{(0,T)\times\Omega}\xi(t,x)d \vert \sqrt{K(x)}Dv_0\vert_\phi \diff t=\int_{(0,T)\times\Omega}\xi(t,x)\gamma(t,x) d\omega(t,x)\\&
   =\int_{(0,T)\times\Omega}\xi(t,x)\phi^{\circ}\left(x,\int_{\partial B_{\phi^\circ(x)}}p d\lambda_{t,x}(p)\right) d\omega(t,x),
\end{align*}
entailing, since $\xi$ is arbitrary, for $\omega$-a.e. $(t,x)\in (0,T)\times\Omega$,
\begin{align*}
   \int_{\partial B_{\phi^\circ(x)}} \phi^{\circ}(x,p)d\lambda_{t,x}(p)= \phi^{\circ}\left(x,\int_{\partial B_{\phi^\circ(x)}}p d\lambda_{t,x}(p)\right),
\end{align*}
and since $\phi^{\circ}(x,\cdot)$ is strictly convex for any fixed $x\in \Omega$, it is well known that $\lambda_{t,x}$ must be a Dirac $\delta$ probability measure, namely it holds $\lambda_{t,x}=\delta\left(\frac{q}{\phi^{\circ}(x,q)}\right)$, for some $q\in \mathbb R^d$. To identify $q$, we come back to \eqref{gam1}, to see that it must be $\gamma\equiv 1$. Then, from \eqref{gam2}, we deduce
$$
\frac{q}{\phi^{\circ}(x,q)}=\frac{\nu_{v_0}(t,x)}{\phi^{\circ}(x,\nu_{v_0}(t,x)},
$$
entailing that whatever $q\in \mathbb R^d$ we choose, we always have that $\lambda_{t,x}=\delta\left(\frac{\nu_{v_0}(t,x)}{\phi^{\circ}(x,\nu_{v_0}(t,x))}\right)$. This concludes the proof.

\section{Numerical simulations}
\label{numerical}
\paragraph{Numerical settings.} To give an idea of the evolution of an anisotropic Cahn-Hilliard equation, we propose some numerical simulations. Here, we consider the following version of the anisotropic Cahn-Hilliard equation used in~\cite{dziwnik2016existence}:
\begin{align*}
\frac{\partial u}{\partial t} &= \Div (m(u)\,\nabla \mu), \\
\mu &= -\eps\,\Div(\,B(\nabla u)\,\nabla u)+\f{1}{\eps} F'(u), 
\end{align*}
where we set $m(u)=1$ for simplicity. $B(\nabla u)$ is a matrix encoding the anisotropy. Here 

\[
B(\nabla u)
 =
 \begin{bmatrix}
   \gamma(\theta)^2 & -\gamma'(\theta)\,\gamma(\theta)\\
   \gamma'(\theta)\,\gamma(\theta) & \gamma(\theta)^2
 \end{bmatrix},
\]
with 
\[
\gamma(\theta) = 1 + G \,\cos(n\,\theta).
\]
Here $\theta$ is the angle between the $x$-axis and the interface normal, whereas $G$ is the strength of the anisotropy. In~\cite{dziwnik2016existence}, the authors show that $G$ must be small enough to prove existence of weak solutions. Observe that $G=0$ corresponds to the isotropic Cahn-Hilliard equation, whereas  $n$ adjusts the symmetry. In two dimensions we have
\[
\theta = \arctan(\partial_y u / \partial_x u).
\]

For simplicity, we update the time derivative  using an explicit Euler scheme:
\[
u^{n+1} = u^n + \Delta t\,  \Div  (m(u^n)\nabla \mu^n),
\]
though one must be cautious with stability. If $\eps$ is small, steep gradients appear in $u$, making the problem stiff. Thus, a sufficiently small time step $\Delta t$ is usually necessary to avoid numerical blow-up. Semi-implicit methods can be used to improve stability. Our goal is to show how the anisotropy can affect the shape of the solutions. Therefore we perform simulations in a range of parameter for $G$. In the following numerical simulations we initialize $u$ randomly in a narrow band around $1/2$. We let the system evolve with time step $\Delta t = 0.001$ and number of steps $N_{T}=5000$ over a grid of size $80\times 80$ with $N_x=100$ grid mesh points on the $x$ dimension and $N_y=100$ points on the $y$ dimension. We plot the simulations for $G=[0.2, 0.5]$ and $\eps=0.7$. 
\begin{figure}[H]
\centering
\includegraphics[scale=0.75]{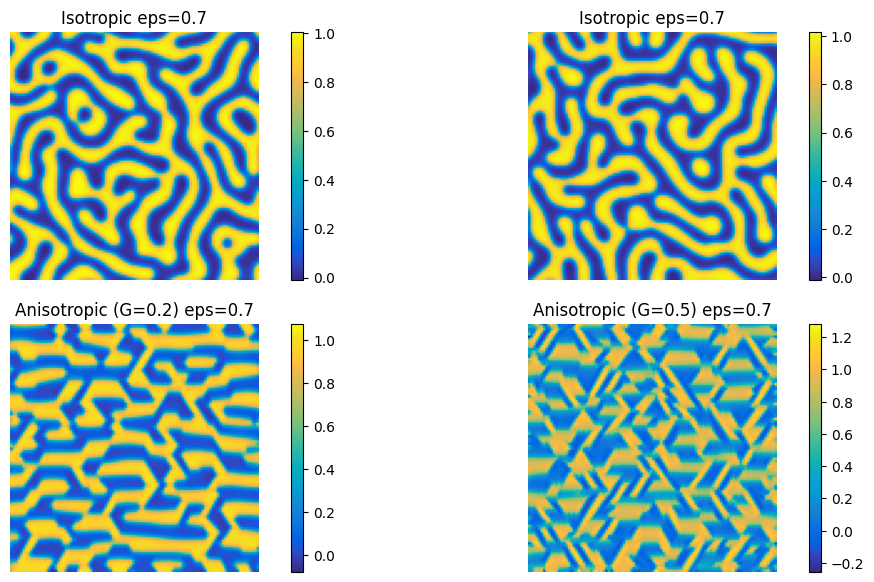}
\caption{Comparison between the isotropic and anisotropic Cahn-Hilliard equation for different values of $G$.}
\end{figure}

\paragraph{Effect of Anisotropy ($G\neq 0$).} 
In the simulations with $G=0$ (isotropic version), the interfaces evolve into  smoothly curved shapes, as one expects from isotropic curvature. When $G\neq 0$, the interfaces have directional biases. For instance, it can develop corners.

\paragraph{Role of Small $\eps$.} 
As $\eps$ becomes smaller, the interface width becomes thinner, so the transition between $u\approx 0$ and $u\approx 1$ becomes sharper. In the numerical simulations, this transition is represented by the thin green line surrounding the domains. Numerically, capturing these sharper interfaces precisely would require thinner spatial meshes or adaptive methods.

\paragraph{Instabilities and Coarsening.}
During the evolution, small domains often shrink and disappear, while larger ones grow, minimizing the total free energy. Anisotropy can accelerate or alter this coarsening process by changing how (anisotropic) curvature forces act along different directions.

\end{appendices}

\bibliographystyle{siam}
\bibliography{merged}

\begin{thebibliography}{10}

\bibitem{alfaromotion}
{\sc M.~Alfaro, D.~Hilhorst, H.~Garcke, H.~Matano, and R.~Sch{\"a}tzle}, {\em Motion by anisotropic mean curvature as sharp interface limit of an inhomogeneous and anisotropic {A}llen--{C}ahn equation. to appear in proc}, Royal Society of Edinburgh A.

\bibitem{alikakos1994convergence}
{\sc N.~D. Alikakos, P.~W. Bates, and X.~Chen}, {\em Convergence of the {C}ahn-{H}illiard equation to the hele-shaw model}, Arch. Ration. Mech. Anal., 128 (1994).

\bibitem{2CNP}
{\sc M.~Amar and G.~Bellettini}, {\em A notion of total variation depending on a metric with discontinuous coefficients}, Ann. Inst. H. Poincar\'e{} C Anal. Non Lin\'eaire, 11 (1994), pp.~91--133.

\bibitem{Ambrosiofuscopallara}
{\sc L.~Ambrosio, N.~Fusco, and D.~Pallara}, {\em Functions of bounded variation and free discontinuity problems}, Oxford university press, 2000.

\bibitem{antonopoulou2018sharp}
{\sc D.~C. Antonopoulou, D.~Bl\"omker, and G.~D. Karali}, {\em The sharp interface limit for the stochastic {C}ahn-{H}illiard equation}, Ann. Inst. Henri Poincar\'e{} Probab. Stat., 54 (2018), pp.~280--298.

\bibitem{barrett2013stable}
{\sc J.~W. Barrett, H.~Garcke, and R.~N{\"u}rnberg}, {\em On the stable discretization of strongly anisotropic phase field models with applications to crystal growth}, ZAMM-Journal of Applied Mathematics and Mechanics/Zeitschrift f{\"u}r Angewandte Mathematik und Mechanik, 93 (2013), pp.~719--732.

\bibitem{barrett2014stable}
\leavevmode\vrule height 2pt depth -1.6pt width 23pt, {\em Stable phase field approximations of anisotropic solidification}, IMA Journal of Numerical Analysis, 34 (2014), pp.~1289--1327.

\bibitem{FragalaBellettini}
{\sc G.~Bellettini and I.~Fragal\`a}, {\em Elliptic approximations of prescribed mean curvature surfaces in {F}insler geometry}, Asymptot. Anal., 22 (2000), pp.~87--111.

\bibitem{Finsler_geometry}
{\sc G.~Bellettini and M.~Paolini}, {\em {Anisotropic motion by mean curvature in the context of Finsler geometry}}, Hokkaido Mathematical Journal, 25 (1996), pp.~537 -- 566.

\bibitem{Bouchitte_Gamma}
{\sc G.~Bouchitte}, {\em Singular perturbations of variational problems arising from a two-phase transition model}, Appl. Math. Optim., 21 (1990), p.~289–314.

\bibitem{boyer}
{\sc F.~Boyer and P.~Fabrie}, {\em Mathematical tools for the study of the incompressible {N}avier-{S}tokes equations and related models}, vol.~183 of Applied Mathematical Sciences, Springer, New York, 2013.

\bibitem{brenier1987decomposition}
{\sc Y.~Brenier}, {\em D{\'e}composition polaire et r{\'e}arrangement monotone des champs de vecteurs}, CR Acad. Sci. Paris S{\'e}r. I Math., 305 (1987), pp.~805--808.

\bibitem{brenier1991polar}
\leavevmode\vrule height 2pt depth -1.6pt width 23pt, {\em Polar factorization and monotone rearrangement of vector-valued functions}, Comm. Pure Appl. Math., 44 (1991), pp.~375--417.

\bibitem{cahn_free_1958}
{\sc J.~W. Cahn and J.~E. Hilliard}, {\em Free {Energy} of a {Nonuniform} {System}. {I}. {Interfacial} {Free} {Energy}}, J. Chem. Phys., 28 (1958), pp.~258--267.

\bibitem{doi:10.1142/S021919972450041X}
{\sc J.~A. Carrillo, C.~Elbar, and J.~Skrzeczkowski}, {\em Degenerate {C}ahn-{H}illiard systems: from nonlocal to local}, Commun. Contemp. Math., 27 (2025), pp.~Paper No. 2450041, 31.

\bibitem{antonio2024competing}
{\sc J.~A. Carrillo, A.~Esposito, C.~Falc{\'o}, and A.~Fern{\'a}ndez-Jim{\'e}nez}, {\em Competing effects in fourth-order aggregation--diffusion equations}, Proceedings of the London Mathematical Society, 129 (2024), p.~e12623.

\bibitem{chen1996global}
{\sc X.~Chen}, {\em Global asymptotic limit of solutions of the {C}ahn-{H}illiard equation}, Journal of Differential Geometry, 44 (1996), pp.~262--311.

\bibitem{CicaleseNagasePisante}
{\sc M.~Cicalese, Y.~Nagase, and G.~Pisante}, {\em The {G}ibbs-{T}homson relation for non homogeneous anisotropic phase transitions}, Advances in Calculus of Variations, 3 (2010), pp.~321--344.

\bibitem{cristini2009nonlinear}
{\sc V.~Cristini, X.~Li, J.~S. Lowengrub, and S.~M. Wise}, {\em Nonlinear simulations of solid tumor growth using a mixture model: invasion and branching}, Journal of mathematical biology, 58 (2009), pp.~723--763.

\bibitem{mr3448925}
{\sc S.~Dai and Q.~Du}, {\em Weak solutions for the {C}ahn-{H}illiard equation with degenerate mobility}, Arch. Ration. Mech. Anal., 219 (2016), pp.~1161--1184.

\bibitem{MR3761096}
{\sc M.~Di~Francesco, A.~Esposito, and S.~Fagioli}, {\em Nonlinear degenerate cross-diffusion systems with nonlocal interaction}, Nonlinear Anal., 169 (2018), pp.~94--117.

\bibitem{dimarino}
{\sc S.~Di~Marino and F.~Santambrogio}, {\em {JKO} estimates in linear and non-linear fokker--planck equations, and keller--segel: {$L^p$} and sobolev bounds}, Ann. Inst. H. Poincar\'e{} C Anal. Non Lin\'eaire, 39 (2022), pp.~1485--1517.

\bibitem{mr2448650}
{\sc J.~Dolbeault, B.~Nazaret, and G.~Savar\'{e}}, {\em A new class of transport distances between measures}, Calc. Var. Partial Differential Equations, 34 (2009), pp.~193--231.

\bibitem{dziwnik2016existence}
{\sc M.~Dziwnik}, {\em Existence of solutions to an anisotropic degenerate {C}ahn-{H}illiard-type equation}, Commun. Math. Sci., 17 (2019), pp.~2035--2054.

\bibitem{dziwnik2017anisotropic}
{\sc M.~Dziwnik, A.~M{\"u}nch, and B.~Wagner}, {\em An anisotropic phase-field model for solid-state dewetting and its sharp-interface limit}, Nonlinearity, 30 (2017), p.~1465.

\bibitem{EPPS}
{\sc C.~Elbar, B.~Perthame, A.~Poiatti, and J.~Skrzeczkowski}, {\em Nonlocal {C}ahn-{H}illiard equation with degenerate mobility: incompressible limit and convergence to stationary states}, Arch. Ration. Mech. Anal., 248 (2024), pp.~Paper No. 41, 38.

\bibitem{crmeca}
{\sc C.~Elbar, B.~Perthame, and J.~Skrzeczkowski}, {\em Pressure jump and radial stationary solutions of the degenerate {Cahn{\textendash}Hilliard} equation}, Comptes Rendus. M\'ecanique, 351 (2023), pp.~375--394.

\bibitem{elliott_cahn-hilliard_1996}
{\sc C.~M. Elliott and H.~Garcke}, {\em On the {C}ahn-{H}illiard equation with degenerate mobility}, SIAM J. Math. Anal., 27 (1996), pp.~404--423.

\bibitem{G2}
\leavevmode\vrule height 2pt depth -1.6pt width 23pt, {\em Diffusional phase transitions in multicomponent systems with a concentration dependent mobility matrix}, Phys. D, 109 (1997), pp.~242--256.

\bibitem{EL}
{\sc C.~M. Elliott and S.~Luckhaus}, {\em A generalized equation for phase separation of a multi-component mixture with interfacial free energy}, IMA Preprint Series \# 887,  (1991).

\bibitem{GGPS}
{\sc C.~G. Gal, M.~Grasselli, A.~Poiatti, and J.~L. Shomberg}, {\em Multi-component {C}ahn-{H}illiard systems with singular potentials: theoretical results}, Appl. Math. Optim., 88 (2023), pp.~Paper No. 73, 46.

\bibitem{GPoiatti}
{\sc C.~G. Gal and A.~Poiatti}, {\em Unified framework for the separation property in binary phase-segregation processes with singular entropy densities}, European J. Appl. Math., 36 (2025), pp.~40--67.

\bibitem{ganedi2024convergence}
{\sc L.~Ganedi, A.~Marveggio, and K.~Stinson}, {\em {C}onvergence of a heterogeneous {A}llen–{C}ahn equation to weighted mean curvature flow}, Adv. Calc. Var.,  (2025).

\bibitem{Patrik}
{\sc H.~Garcke, P.~Knopf, and A.~Signori}, {\em The anisotropic {C}ahn--{H}illiard equation with degenerate mobility: existence of weak solutions}, arXiv preprint arXiv:2502.13799.

\bibitem{garcke2023anisotropic}
{\sc H.~Garcke, P.~Knopf, and J.~Wittmann}, {\em The anisotropic {C}ahn-{H}illiard equation: regularity theory and strict separation properties}, Discrete Contin. Dyn. Syst. Ser. S, 16 (2023), pp.~3622--3660.

\bibitem{garcke2009anisotropic}
{\sc H.~Garcke and C.~Kraus}, {\em An anisotropic, inhomogeneous, elastically modified {G}ibbs-{T}homson law as singular limit of a diffuse interface model}, Adv. Math. Sci. Appl., 20 (2010), pp.~511--545.

\bibitem{garcke2011existence}
{\sc H.~Garcke and S.~Schaubeck}, {\em Existence of weak solutions for the {S}tefan problem with anisotropic {G}ibbs-{T}homson law}, Adv. Math. Sci. Appl., 21 (2011), pp.~255--283.

\bibitem{G1}
{\sc H.~Garcke, B.~Stoth, and B.~Nestler}, {\em Anisotropy in multi-phase systems: a phase field approach}, Interfaces Free Bound., 1 (1999), pp.~175--198.

\bibitem{gurtin1987some}
{\sc M.~E. Gurtin}, {\em Some results and conjectures in the gradient theory of phase transitions}, in Metastability and incompletely posed problems, Springer, 1987, pp.~135--146.

\bibitem{hensel2021bv}
{\sc S.~Hensel and T.~Laux}, {\em B{V} solutions for mean curvature flow with constant contact angle: {A}llen-{C}ahn approximation and weak-strong uniqueness}, Indiana Univ. Math. J., 73 (2024), pp.~111--148.

\bibitem{MR1617171}
{\sc R.~Jordan, D.~Kinderlehrer, and F.~Otto}, {\em The variational formulation of the {F}okker-{P}lanck equation}, SIAM J. Math. Anal., 29 (1998), pp.~1--17.

\bibitem{kim2024density}
{\sc I.~Kim, A.~Mellet, and Y.~Wu}, {\em Density-constrained chemotaxis and {H}ele-{S}haw flow}, Transactions of the American Mathematical Society, 377 (2024), pp.~395--429.

\bibitem{KOBAYASHI1993410}
{\sc R.~Kobayashi}, {\em Modeling and numerical simulations of dendritic crystal growth}, Physica D: Nonlinear Phenomena, 63 (1993), pp.~410--423.

\bibitem{laux_kroemer}
{\sc M.~Kroemer and T.~Laux}, {\em The {H}ele–{S}haw flow as the sharp interface limit of the {C}ahn–{H}illiard equation with disparate mobilities}, Comm. Partial Differential Equations, 47 (2022), pp.~2444--2486.

\bibitem{lam2019sharp}
{\sc K.~F. Lam}, {\em Sharp interface limit of a non-mass-conserving {C}ahn--{H}illiard system with source terms and non-solenoidal darcy flow}, arXiv preprint arXiv:1902.07840,  (2019).

\bibitem{LAUT:2020}
{\sc T.~Laux}, {\em A gradient-flow approach for the convergence of the anisotropic {A}llen-{C}ahn equation}, RIMS Kôkyûroku, Geometric Aspects of Solutions to Partial Differential Equations, 2172 (2020).

\bibitem{laux2018convergence}
{\sc T.~Laux and T.~M. Simon}, {\em Convergence of the {A}llen-{C}ahn equation to multiphase mean curvature flow}, Comm. Pure Appl. Math., 71 (2018), pp.~1597--1647.

\bibitem{LauxUllrich}
{\sc T.~Laux, K.~Stinson, and C.~Ullrich}, {\em Diffuse-interface approximation and weak–strong uniqueness of anisotropic mean curvature flow}, European J. Appl. Math.,,  (2022).

\bibitem{lee2016sharp}
{\sc A.~A. Lee, A.~Munch, and E.~Suli}, {\em Sharp-interface limits of the {C}ahn--{H}illiard equation with degenerate mobility}, SIAM Journal on Applied Mathematics, 76 (2016), pp.~433--456.

\bibitem{lisini-ch-gradient-flow}
{\sc S.~Lisini, D.~Matthes, and G.~Savar\'{e}}, {\em Cahn-{H}illiard and thin film equations with nonlinear mobility as gradient flows in weighted-{W}asserstein metrics}, J. Differential Equations, 253 (2012), pp.~814--850.

\bibitem{Luckhaus1989TheGR}
{\sc S.~Luckhaus and L.~Modica}, {\em The {G}ibbs-{T}hompson relation within the gradient theory of phase transitions}, Arch. Ration. Mech. Anal., 107 (1989), pp.~71--83.

\bibitem{luckhaus1995implicit}
{\sc S.~Luckhaus and T.~Sturzenhecker}, {\em Implicit time discretization for the mean curvature flow equation}, Calc. Var. Partial Differential Equations, 3 (1995), pp.~253--271.

\bibitem{MZ}
{\sc A.~Miranville and S.~Zelik}, {\em Robust exponential attractors for {C}ahn-{H}illiard type equations with singular potentials}, Math. Methods Appl. Sci., 27 (2004), pp.~545--582.

\bibitem{Modica}
{\sc L.~Modica and S.~Mortola}, {\em Un esempio di {$\Gamma$}-convergenza}, Boll. Un. Mat. Ital. B (5), 14 (1977), pp.~285--299.

\bibitem{Matano}
{\sc K.-I. Nakamura, H.~Matano, D.~Hilhorst, and R.~Sch\"atzle}, {\em Singular limit of a reaction-diffusion equation with a spatially inhomogeneous reaction term}, J. Statist. Phys., 95 (1999), pp.~1165--1185.

\bibitem{Qi}
{\sc Y.~Qi and G.-F. Zheng}, {\em Convergence of solutions of the weighted {A}llen-{C}ahn equations to {B}rakke type flow}, Calc. Var. Partial Differential Equations, 57 (2018), pp.~Paper No. 133, 41.

\bibitem{V1}
{\sc A.~R\"atz, A.~Ribalta, and A.~Voigt}, {\em Surface evolution of elastically stressed films under deposition by a diffuse interface model}, J. Comput. Phys., 214 (2006), pp.~187--208.

\bibitem{MR2005609}
{\sc R.~Rossi and G.~Savar\'{e}}, {\em Tightness, integral equicontinuity and compactness for evolution problems in {B}anach spaces}, Ann. Sc. Norm. Super. Pisa Cl. Sci. (5), 2 (2003), pp.~395--431.

\bibitem{santambrogio2015optimal}
{\sc F.~Santambrogio}, {\em Optimal transport for applied mathematicians}, Birk{\"a}user, NY, 55 (2015), p.~94.

\bibitem{CahnTaylor}
{\sc J.~E. Taylor and J.~W. Cahn}, {\em Linking anisotropic sharp and diffuse surface motion laws via gradient flows}, J. Statist. Phys., 77 (1994), pp.~183--197.

\bibitem{V2}
{\sc S.~Torabi, J.~Lowengrub, A.~Voigt, and S.~Wise}, {\em A new phase-field model for strongly anisotropic systems}, Proc. R. Soc. Lond. Ser. A Math. Phys. Eng. Sci., 465 (2009), pp.~1337--1359.
\newblock With supplementary material available online.

\end{thebibliography}

\end{document}